\documentclass[sn-mathphys,Numbered]{sn-jnl}

\usepackage[utf8]{inputenc}
\usepackage{amsmath}
\usepackage{amsfonts}
\usepackage{amssymb}
\usepackage{bm}
\usepackage{amsthm}
\usepackage{stmaryrd}
\usepackage{graphicx}
\usepackage{bbm}
\usepackage{ulem}
\usepackage{caption}
\usepackage{subcaption}
\usepackage{multirow}%
\usepackage{mathrsfs}%
\usepackage[title]{appendix}%
\usepackage{xcolor}%
\usepackage{textcomp}%
\usepackage{manyfoot}%
\usepackage{booktabs}%
\usepackage{algorithm}%
\usepackage{algorithmicx}%
\usepackage{algpseudocode}%
\usepackage{listings}%
\usepackage{float}
\usepackage{pgfplots}
\pgfplotsset{compat=1.9}

\theoremstyle{plain}
\newtheorem{theorem}{Theorem}[section]
\newtheorem{lemma}[theorem]{Lemma}
\newtheorem{proposition}[theorem]{Proposition}
\newtheorem{assumption}[theorem]{Assumption}
\theoremstyle{definition}
\newtheorem{definition}[theorem]{Definition}
\newtheorem{remark}[theorem]{Remark}
\theoremstyle{plain}

\DeclareMathOperator*{\esssup}{ess\,sup}
\DeclareMathOperator*{\essinf}{ess\,inf}
\DeclareMathOperator*{\supp}{supp}

\DeclareMathOperator{\id}{id}
\DeclareMathOperator{\dive}{div}

\usepackage{xcolor}

\newcommand{\reviewun}[1]{#1}
\newcommand{\reviewde}[1]{#1}

\newcommand{\secondereview}[1]{#1}

\begin{document}

\title[A moment approach for entropy solutions of parameter-dependent hyperbolic conservation laws]{A moment approach for entropy solutions of parameter-dependent hyperbolic conservation laws}

\author*[1]{\fnm{Clément} \sur{Cardoen}}\email{clement.cardoen@ec-nantes.fr}

\author[2]{\fnm{Swann} \sur{Marx}}\email{swann.marx@ls2n.fr}

\author[1]{\fnm{Anthony} \sur{Nouy}}\email{anthony.nouy@ec-nantes.fr}

\author[3]{\fnm{Nicolas} \sur{Seguin}}\email{nicolas.seguin@inria.fr}

\affil*[1]{Nantes Université, Centrale Nantes, Laboratoire de Mathématiques Jean Leray, CNRS UMR 6629, France}

\affil[2]{Nantes Université, Centrale Nantes, LS2N, CNRS UMR 6004, France}

\affil[3]{Inria, Centre de l'Université Côte d'Azur, antenne de Montpellier,  Imag, UMR CNRS 5149, Université de Montpellier, France}

\abstract{
We propose a numerical method to solve parameter-dependent \reviewde{scalar} hyperbolic partial differential equations (PDEs) with a moment approach, based on a previous work from Marx et al. (2020). This approach relies on a very weak notion of solution of nonlinear equations, namely parametric entropy measure-valued (MV) solutions, satisfying linear equations in the space of Borel measures. The infinite-dimensional linear problem is approximated by a hierarchy of convex, finite-dimensional, semidefinite programming problems, called Lasserre's hierarchy. This gives us a sequence of approximations of the moments of the occupation measure associated with the parametric entropy MV solution, which is proved to converge. In the end, several post-treatments can be performed from this approximate moments sequence. In particular, the graph of the solution can be reconstructed from an optimization of the Christoffel-Darboux kernel associated with the approximate measure, that is a powerful approximation tool able to capture a large class of irregular functions. Also, for uncertainty quantification problems, several quantities of interest can be estimated, sometimes directly such as the expectation of smooth functionals of the solutions. The performance of our approach is evaluated through numerical experiments on the inviscid Burgers equation  with parametrised initial conditions or parametrised flux function.
}


\maketitle

\section*{Acknowledgments}

 The authors benefit from the support of the French National Research Agency within the projects ANR-20-THIA-0011 (AIBY4) and ANR-22-CE40-0010. The first and third authors thank the France 2030 framework programme Centre Henri Lebesgue ANR-11-LABX-0020-01 for its stimulating mathematical research programs.


\section{Introduction}
Non-linear hyperbolic conservation laws model numerous physical phenomena in fluid mechanics, traffic flow or non-linear acoustics \cite{dafermos10,whitham11}. The numerical computation of such equations is often a challenge since their solutions may present discontinuities, even if the initial data are smooth. Numerous numerical methods exist to approximate them, amongst which we may cite finite volume or finite difference schemes \cite{leveque92} or the front-tracking method \cite{holden15}. We are interested in this paper in solving parameter-dependent hyperbolic conservation laws, which are considered for various tasks in data assimilation \cite{boulanger15}, uncertainty quantification \cite{abgrall17, poette09, bijl13, zhong22, badwaik21}, sensitivity analysis \cite{chalons18}, or error analysis \cite{giesselmann20}. The parameters in our context appear in the initial data and in the flux function and are associated with a probability measure. The computation of approximate solutions for many instances of the parameters is usually prohibitive and require reduced order models.

Model order reduction methods aim at providing an approximation of the solution $u(z,\xi)$, depending on physical variables $z$ and parameters $\xi$, that can be efficiently evaluated. They either rely on an explicit approximation of the solution map $\xi \mapsto u(\cdot,\xi)$ or an approximation of the solution manifold $\{u(\cdot,\xi) : \xi \in \Xi\}$  by some dimension reduction method.  The main challenge for models driven by conservation laws is that the solution maps and solution manifolds are highly nonlinear (in particular due to the presence of discontinuities), that require the introduction of nonlinear approximation or dimension reduction methods. Several model reduction methods based on compositions have been proposed, that include methods based on parameter-dependent changes of variables \cite{reiss18,grundel22} or deep learning methods using neural networks \cite{laakman21}. These methods usually require high computational resources and huge training data for the approximation of highly nonlinear solution maps. 

Here, we follow a different approach and propose a new surrogate modelling method. It is an extension of \cite{marx1} to parameter-dependent or random conservation laws. Whereas it is classical to seek entropy weak solutions to hyperbolic conservation laws \cite{dafermos10,kruzhkov70}, we are rather interested in so-called  entropy measure-valued (MV) solutions, an even weaker notion of solution, introduced  by DiPerna in \cite{diperna85,necas_weak_1996}. To a MV solution corresponds an occupation measure, whose marginal is the MV solution. 
Even if this notion of solution is very weak, there is a correspondence with entropy weak solution. The measure concentrated on the graph of the entropy weak solution is a MV solution. 
It is worth noting that the formulation in the setting of MV solutions leads to a linear problem, so that some efficient tools from convex analysis can be applied.

We start with a theoretical framework for parameter-dependent conservation laws similar to the one of \cite{mishra12,mishra14}. However, in our case, we introduce a weak-parametric formulation of the problem, where the classical entropy weak formulation is also integrated with respect to the parameter. The purpose of this formulation is to obtain an equivalent definition of parameter-dependent entropy MV solutions using the moments of the associated occupation measure with respect to all the variables, including the parameters. Under the assumption that flux function is polynomial and that the initial data can be described by semi-algebraic functions, the entropy formulation becomes a set of linear constraints on the moments of the occupation measure and we can follow the procedure initiated in \cite{marx1}. Indeed, this allows us to consider the problem as a generalized moment problem (GMP), an infinite-dimensional optimization problem over sequences of moments of measures, where both the cost and the constraints are linear with respect to the moments of the measures. Powerful results from real algebraic geometry allow to reformulate the constraint that a sequence is a moment sequence into tractable semi-definite constraints. This problem is then solved using Lasserre's (moment sum-of-squares) hierarchy \cite{lasserre09}, which consists in solving a sequence of convex semi-definite programs of increasing size to approximate the moments of the occupation measure. Note that the use of Lasserre's hierarchy for solving PDEs has been also recently considered in \cite{henrion2023infinite}, although  with a different approach where the considered measure is defined on an infinite-dimensional function space, and assumed to be concentrated on the solution of the PDE. 

Obtaining an approximation of the moments can be costly, but once this offline computation is performed, efficient online post-treatments are possible. 
First, we can naturally obtain expectations of variables of interest that are functions of the moments of the solution.  
Also, the graph of the entropy weak solution (for any parameter value) can be recovered using a localization property of the Christoffel-Darboux kernel of the approximate occupation measure, following the methodology proposed in \cite{marx2}. This powerful approximation method allows to capture efficiently discontinuities in the solutions.  Using the moment completion technique from \cite{henrion20}, one can also have access to other  quantities of interest, such as statistical moments of point-wise evaluations of the solution. 


\paragraph{Outline}
This paper is organized as follows. We first introduce some notations and the problem considered. Section \ref{sec:notOfSol} introduces different notions of solutions for parametrised scalar conservation laws and examines some links between these notions. Section \ref{sec:GMPAndRel} introduces the moment-SOS hierarchy and indicates how to perform several post-treatments such as retrieving the graph of the solution or estimating   statistical moments of the solution. Finally, Section \ref{sec:numExa} presents some numerical experiments. 

\subsection{Notations}
For $\mathcal X\subset\mathbb R^n$, with $n\in\mathbb N$, let $\mathcal{C}(\mathcal X)$, $\mathcal C_0(\mathcal X)$ and $\mathcal{C}_c^1(\mathcal{X})$ denote the space of functions on $\mathcal{X}$ that are continuous, continuous and vanishing at infinity and continuously differentiable with compact support, respectively. 
The sets of signed Borel measures and positive Borel measures are denoted $\mathcal{M}(\mathcal{X})$ and $\mathcal{M}(\mathcal{X})_+$, respectively. The set of probability measures on $\mathcal{X}$ is denoted by $\mathcal{P}(\mathcal{X})$.
The measure $\lambda_\mathcal X\in\mathcal{M}(\mathcal{X})_+$ denotes the Lebesgue measure on $\mathcal{X}$, and for $ B\subset\mathcal X$ a Borel set, $\vert B \vert$ denotes its Lebesgue measure.
Given a vector $\mathbf{w}=(w_1,\hdots, w_n)$, we denote by $\mathbb R[\mathbf w]$ the ring of real multivariate polynomials in the variable $w_1,\hdots,w_n$, and for a multi-index $\bm \alpha = (\alpha_1,\hdots,\alpha_n)$, $\mathbf w^{\bm \alpha} := w_1^{\alpha_1}\hdots w_n^{\alpha_n}$. 
Given a positive Borel measure $\mu$, we denote by $\supp(\mu)$ its support, defined as the smallest closed set whose complement has measure zero.

\subsection{Definition of the problem}
We consider parameter-dependent scalar hyperbolic conservation laws that are formulated as a Cauchy problem
\begin{subequations}\label{eq:parPDE}
\begin{align}
&\partial_t u(t,\mathbf x,{\bm\xi})+\dive_{\mathbf x}\mathbf f(u(t,\mathbf x,\bm\xi),\bm\xi)=0,\quad(t,\mathbf x,{\bm\xi})\in\mathbb R_+\times\mathbb R^n\times\mathbf{\mathbf{\Xi}},\label{subeq:parPDE}\\
&u(0,\mathbf x,{\bm\xi})=u_0(\mathbf x,{\bm\xi}),\quad (\mathbf x,{\bm\xi})\in\mathbb R^n\times \mathbf{\Xi},
\end{align}
\end{subequations}
where $t \in \mathbb R_+$ is the time variable, $\mathbf{x} \in \mathbb{R}^n$ is the space variable, and where $\bm\xi$ is a parameter in a parameter set $\mathbf\Xi\subset\mathbb R^p$, $p\in\mathbb N$. Then data are the flux $\mathbf f : \mathbb R \times \mathbf\Xi \to \mathbb R^n$  and the initial condition $u_0 : \mathbb R^n\times\mathbf\Xi\to \mathbb R$.

The parameter set $\mathbf\Xi$ is equipped with a probability measure $\rho$. We assume that $\mathbf\Xi$ is compact and the initial condition $u_0 \in\mathcal L^\infty(\mathbf\Xi;\mathcal L^\infty(\mathbb R^n))$. Moreover, we shall make the following assumptions on $\mathbf f$.
\begin{assumption}\label{ass:assOnf}
    For all $K\subset\mathbb R$ compact, there exists $C_K\in\mathbb R$ such that for all $u\in K$ and $\rho-\text{a.e.}$, $\Vert\mathbf f(u,\bm\xi)\Vert\leq C_K$.
    Moreover, $\rho-\text{a.e.}$, $\mathbf f(\cdot,\bm\xi)\in\mathcal C^1(\mathbb R;\mathbb R^n)$.
\end{assumption}
This assumption is satisfied for a polynomial function $f$, that will be assumed in the next section. 
Then for the sake of simplicity, we restrict our analysis to this setting. A more general setting can be found in \cite{mishra14}.

Note that the initial data $u_0$ and the flux $f$ may depend on distinct parameters but for the sake of clarity, and without loss of generality, we indicate a dependence on the same set of parameters $\bm \xi$.


\section{Notions of parameter-dependent solution}\label{sec:notOfSol}
\subsection{Parametric entropy solution}
We start by introducing the notion of parametric entropy weak solution, that is defined point-wise in the parameter domain. This may be considered as a strong-parametric solution, that is a straightforward notion of solution when a parameter is considered, see, e.g., \cite{mishra12} or \cite{mishra14}.

\begin{definition}[Entropy pairs]\label{def:entPai}
    Let $\eta$ be a locally Lipschitz and convex function from $\mathbb R$ to $\mathbb R$. Let $\mathbf q:\mathbb R\times\mathbf\Xi\rightarrow\mathbb R^n$ such that $\partial_u\mathbf q(u,\bm\xi)=\eta'(u)\partial_u\mathbf f(u,\bm\xi)$ for $\rho$-almost $\bm \xi$ and almost all $u\in\mathbb R$. Then $(\eta,\mathbf q)$ is called an entropy pair associated with conservation law~\eqref{subeq:parPDE}.
\end{definition}

We may notice that for an entropy pair $(\eta,\mathbf q)$, for $\rho$-almost all $\bm\xi$, $\mathbf q(\cdot,\bm\xi)$ is a locally Lipschitz function.

We now introduce three specific families of entropy pairs, each of them having a particular theoretical or numerical objective.
\begin{definition}[$\mathcal C^1$ family of entropy pairs]
    The $\mathcal C^1$ family of entropy pairs, denoted $\mathcal E_{\mathcal C}$, is defined as the set of entropy pairs $(\eta,\mathbf q)$ such that $\eta\in\mathcal C^1(\mathbb R)$ and for $\rho$-almost all $\bm\xi$, $\mathbf q(\cdot,\bm\xi)\in\mathcal C^1(\mathbb R)$.
\end{definition}
Note under assumption \ref{ass:assOnf}, if $(\eta,\mathbf q)$ is an entropy pair with $\eta\in \mathcal{C}^1$, then $(\eta,\mathbf q) \in\mathcal E_{\mathcal C} $.
The $\mathcal C^1$ family of entropy pairs is related to the (opposite of the) thermodynamic entropy and to the second law of thermodynamics, for fluid dynamics models. The conservation law~\eqref{subeq:parPDE}, for a fixed $\bm\xi$, can be seen as a simplification of such models.

\begin{definition}[Kruzhkov family of entropy pairs]\label{def:KruEntPai}
    The Kruzhkov family of entropy pairs is defined by
    \begin{equation*}\label{eq:famEntPai}
    \mathcal E_K:=\left\{(\eta_v,\mathbf q_v) : v\in\mathbb R\right\},
    \end{equation*}
    where for all $v\in\mathbb R$, for all $u\in\mathbb R$ and for $\rho$-almost all $\bm\xi$, $\eta_v(u):=\vert u-v\vert$ and $\mathbf q_v(u,\bm\xi):=\text{sign}(u-v)(\mathbf f(u,\bm\xi)-\mathbf f(v,\bm\xi))$.
\end{definition}

Compared to $\mathcal E_{\mathcal C}$, the family $\mathcal E_K$ has the advantage of being explicitly described and of carrying strong results coming from Kruzhkov's fundamental paper \cite{kruzhkov70}, allowing to obtain some theoretical results, such as uniqueness and stability.


\begin{definition}[Polynomial family of entropy pairs]
    The polynomial family of entropy pairs $\mathcal E_P$ is defined as the set of entropy pairs $(\eta,\mathbf q)$ such that $\eta$ is a polynomial function. If $\mathbf f$ is a polynomial function, then, for $\rho$-almost all $\bm\xi$, $\mathbf q(\cdot,\bm\xi)$ is also a polynomial function.
\end{definition}

In the case of a uniformly convex flux function, a single polynomial entropy can be sufficient to select the relevant solution (see e.g. \cite{delellis04,krupa19}).
Actually, our motivation is different.
In numerical experiments, we shall use subsets of the polynomial family of entropy pairs. The SOS-moment (Lasserre's) hierarchy, later exposed in this paper, relies indeed on a polynomial setting. Although it is possible to implement our numerical method with $\mathcal E_K$ as in \cite{marx1}, it is easier to do so with subsets of $\mathcal E_P$ when possible.

\begin{definition}[Parametric entropy solution]\label{def:parEntSol}
Consider a family of entropy pairs $\mathcal E$.
Let $u_0 : \mathbb R^n \times \mathbf \Xi \to \mathbb R$ such that 
 $u_0 (\cdot, \bm \xi) \in\mathcal L^\infty(\mathbb R^n)$ $\rho$-almost everywhere and $\mathbf f$ satisfy Assumption \ref{ass:assOnf}. A function $u : \mathbb R_+\times\mathbb R^n \times \mathbf \Xi \to \mathbb R $ such that for $\rho$-almost all $\bm\xi$, $u(\cdot,\cdot,\bm\xi)\in\mathcal L^\infty(\mathbb R_+\times\mathbb R^n)$ is a parametric entropy solution for $\mathcal E$ if, for all $(\eta,\mathbf q)\in\mathcal E$, for all non-negative test functions $\psi\in\mathcal C_c^1(\mathbb R_+\times\mathbb R^n)$ and $\rho$-almost all $\bm\xi$, it satisfies
\begin{equation}\label{eq:entSol}
\int_{\mathbb R_+}\int_{\mathbb R^n}\left(\partial_t\psi\eta(u)+\nabla_{\mathbf x}\psi\cdot\mathbf q(u,\bm\xi)\right)d\mathbf xdt+\int_{\mathbb R^n}\psi\eta(u_0)d\mathbf x\geq 0.
\end{equation}
\end{definition}

\begin{proposition}\label{pro:equiEkEC}
    A function $u$ is a parametric entropy solution for $\mathcal E_K$ if and only if it is a parametric entropy solution for $\mathcal E_{\mathcal C}$.
\end{proposition}
\begin{proof}
For the proof, see Lemma 4.1 in \cite{godlewski91} and the discussion which follows.  
\end{proof}

\begin{theorem}\label{th:parametric-entropy-weak-solution}
    If, for $\rho$-almost all $\bm\xi$, the initial data $u_0(\cdot,\bm\xi)  \in\mathcal L^\infty(\mathbb R^n)$ and if $\mathbf f$ satisfies Assumption \ref{ass:assOnf}, then problem \eqref{eq:parPDE} has a unique parametric entropy solution $u$ for $\mathcal E_K$, or equivalently, $\mathcal E_{\mathcal C}$. Moreover it satisfies for all $t\in\mathbb R_+$ and for $\rho$-almost all $\bm\xi$,
    \begin{equation}\label{eq:staPro}\Vert u(t,\cdot,\bm\xi)\Vert_{\mathcal L^\infty(\mathbb R^n)}\leq\Vert u_0 (\cdot , \bm \xi)\Vert_{\mathcal L^\infty(\mathbb R^n)}.
    \end{equation}
\end{theorem}
\begin{proof}
    From \cite[Theorem 5.2]{godlewski91}, we have that, for $\rho$-almost all $\bm\xi$, there exists a unique solution $u(\cdot,\cdot,\bm\xi)\in\mathcal L^\infty(\mathbb R_+\times\mathbb R^n)$ for $\mathcal E_{\mathcal C}$. Then, from Proposition \ref{pro:equiEkEC}, we have that, for $\rho$-almost all $\bm\xi$, there exists a unique solution $u(\cdot,\cdot,\bm\xi)\in\mathcal L^\infty(\mathbb R_+\times\mathbb R^n)$ for $\mathcal E_K$. The stability property \eqref{eq:staPro} is deduced  from \cite[Theorem 6.2.7]{dafermos10}.
\end{proof}
\begin{remark}\label{rem:bonOf_u}
Note that Theorem \ref{th:parametric-entropy-weak-solution} does not provide any information on the 
measurability of $u$.
Under the assumption that $\mathbf\Xi\ni\bm\xi\mapsto u(\cdot,\cdot,\bm\xi)\in\mathcal L^\infty(\mathbb R_+\times\mathbb R^n)$ is Bochner measurable and that $u_0\in\mathcal L^\infty(\mathbf\Xi;\mathcal L^\infty(\mathbb R^n))$, Theorem \ref{th:parametric-entropy-weak-solution} allows to deduce 
\begin{equation*}
    \Vert u \Vert_{\mathcal L^\infty(\mathbb R_+\times\mathbb R^n \times \mathbf\Xi)} \le \Vert u_0 \Vert_{\mathcal L^\infty(\mathbb R^n \times \mathbf\Xi)} .
\end{equation*}
We refer to \cite[Theorem 3.3]{mishra12} that provides measurability properties of $u$ under additional assumptions on $u_0$.
\end{remark}

\begin{remark}
    We might hope that imposing that $u_0\in\mathcal L^\infty(\mathbf\Xi;\mathcal L^\infty(\mathbb R^n))$ may be sufficient to have that $\bm\xi\mapsto u(\cdot,\cdot,\bm\xi)$ is Bochner measurable, but this has not been proved yet.
\end{remark}


\subsection{Weak-parametric entropy solutions}
The next notion of solution is weaker. While the parametric entropy solution adopts a pointwise point of view in the parameter domain, the following notion of solution is deduced by integration over the parameter domain.
\begin{definition}[Weak-parametric entropy solution]
Consider a family of entropy pairs $\mathcal E$.
Let $u_0\in\mathcal L^\infty(\mathbb R^n \times \mathbf\Xi)$ and $\mathbf f$ satisfy Assumption \ref{ass:assOnf}. A measurable function $u : \mathbb R_+\times\mathbb R^n\times\mathbf\Xi \to \mathbb R$ in $\mathcal L^\infty(\mathbb R_+\times\mathbb R^n\times\mathbf\Xi)$ is called a weak-parametric entropy solution for $\mathcal E$ if, for all $(\eta,\mathbf q)\in\mathcal E$ and all non-negative test functions $\phi\in\mathcal C(\mathbf\Xi;\mathcal C^1_c(\mathbb R_+\times\mathbb R^n))$, it satisfies 
    \begin{equation}\label{eq:ent-parSol}
        \int_{\mathbb R_+}\int_{\mathbb R^n}\int_{\mathbf\Xi}\left(\partial_t\phi\eta(u)+\nabla_{\mathbf x}\phi\cdot\mathbf q(u,\bm\xi)\right)d\rho(\bm\xi)d\mathbf xdt+\int_{\mathbb R^n}\int_{\mathbf\Xi}\phi\eta(u_0)d\rho(\bm\xi)d\mathbf x\geq 0.
    \end{equation}
\end{definition}

It is at first glance a weaker notion of solution, but 
we shall see that under certain assumptions, 
both notions of parametric entropy solution and weak-parametric entropy solution coincide.

\begin{theorem}\label{th:forEqu}
Assume that $u_0 \in \mathcal L^\infty(\mathbf\Xi,\mathcal L^\infty(\mathbb R^n))$ and $\mathbf f$ satisfies Assumption \ref{ass:assOnf}. A function $u$, such that $\bm\xi\mapsto u(\cdot,\cdot,\bm\xi)$ is Bochner measurable, is a  parametric entropy solution for $\mathcal E_K$ if and only if it is a weak-parametric entropy solution for $\mathcal E_K$.

   %
\end{theorem}
\begin{proof}
Let $u$  be a parametric entropy solution for $\mathcal E_K$, and let $v\in\mathbb R$ and $\phi\in\mathcal C(\mathbf\Xi;\mathcal C^1_c(\mathbb R_+\times\mathbb R^n))$. 
From Remark \ref{rem:bonOf_u}, and since $u_0 \in \mathcal L^\infty(\mathbf\Xi,\mathcal L^\infty(\mathbb R^n))$ and $\bm\xi\mapsto u(\cdot,\cdot,\bm\xi)$ is Bochner measurable, we have that $u\in\mathcal L^\infty(\mathbb R_+\times\mathbb R^n\times\mathbf\Xi)$. For $\rho$-almost all $\bm\xi$, $\phi(\cdot,\cdot,\bm\xi)\in\mathcal C^1_c(\mathbb R_+\times\mathbb R^n)$, thus $u$ verifies equation \eqref{eq:entSol} for $\psi= \phi(\cdot,\cdot,\bm\xi)$. Let us integrate equation \eqref{eq:entSol} on $\mathbf\Xi$. First, let us consider the terms where $\eta_v$ appears. From Remark \ref{rem:bonOf_u}, $u$ is essentially bounded. Since $\eta_v$ is continuous, $\eta_v\circ u$ is also essentially bounded. Since $\phi$ and its derivative are continuous in $\bm\xi$ and $\mathbf\Xi$ is a compact set, the terms where $\eta_v$ appears are integrable in $\bm\xi$. Recalling the definition of $\mathbf q_v$ in Definition \ref{def:KruEntPai}, we have that, for all $y\in\mathbb R$, $\rho$-almost everywhere, $\Vert\mathbf q_v(y,\bm\xi)\Vert\leq\Vert\mathbf f(y,\bm\xi)\Vert+\Vert\mathbf f(v,\bm\xi)\Vert$. From Assumption \ref{ass:assOnf}, and since from the same argument as for the terms where $\eta_v$ appears, $u$ is essentially bounded, there exists $C,C_v\in\mathbb R$ such that $\Vert\mathbf q_v(u,\bm\xi)\Vert\leq C+C_v$, $\rho$-almost everywhere. Thus, $\bm\xi\mapsto\int_{\mathbb R_+}\int_{\mathbb R^n}\nabla_{\mathbf x}\phi\cdot\mathbf q(u,\bm\xi)d\mathbf xdt$ is integrable and integrating on $\mathbf\Xi$ yields equation \eqref{eq:ent-parSol}.
    \\ 
    Conversely, let $u\in\mathcal L^\infty(\mathbb R_+\times\mathbb R\times\mathbf\Xi)$ be a weak-parametric entropy solution, $v\in\mathbb R$ and $\phi\in\mathcal C(\mathbf\Xi;\mathcal C^1_c(\mathbb R_+\times\mathbb R^n))$ such that $\phi(t,\mathbf x,\bm\xi):=\psi(t,\mathbf x)\gamma(\bm\xi)$ where $\psi\in\mathcal C^1_c(\mathbb R_+\times\mathbb R^n)$ and $\gamma\in\mathcal C(\mathbf\Xi)$. The function $u$ then verifies (\ref{eq:ent-parSol}) for our particular choice of $\phi$. Inequality \eqref{eq:ent-parSol} can be rewritten as $\int_{\mathbf\Xi}I(\bm\xi)\gamma(\bm\xi)d\rho(\bm\xi)\geq0$. Since $I$ is $\rho$-measurable, $\mathbbm 1_{I<0}$ is $\rho$-measurable. Moreover, from \cite[Theorem 12.7]{aliprantis06}, since $\rho$ is a finite Borel measure and $\mathbf\Xi$ is a Polish space, we have that $\rho$ is a regular measure. Thus, there exists a sequence $\gamma_n\in\mathcal C(\mathbf\Xi)$ such that $\Vert\mathbbm 1_{I<0}-\gamma_n\Vert_{\mathcal L^1(\mathbf\Xi)}\to 0$ as $n\to \infty$ and 
      $\int_{\mathbf\Xi}I(\bm\xi)\mathbbm 1_{I<0}(\bm\xi)d\rho(\bm\xi)\geq0$. Yet, $I\mathbbm 1_{I<0}\leq0$ $\rho$-almost everywhere. Thus, $I\geq0$ $\rho$-almost everywhere and it gives us that $u$ is a parametric entropy solution, which concludes  the proof.
 
\end{proof}

\subsection{Measure-valued solutions}

Following DiPerna~\cite{diperna85}, previous notions of solutions are extended to the weaker case of measure-valued solutions thanks to the notion of Young measure.
\begin{definition}[Young measure]
A Young measure on a Euclidean space $\mathcal X$ is a map $\mu:\mathcal X\rightarrow\mathcal P(\mathbb R),\tau\mapsto\mu_\tau$ such that for all $g\in\mathcal C_0(\mathbb R)$ the function $\tau\mapsto\int_{\mathbb R}g(y)\mu_\tau(dy)$ is measurable.
\end{definition}
From this, we can seek an even weaker notion of solution that is 
a Young measure $\mu_{(t,\mathbf x,\bm\xi)}$ which satisfies the following Cauchy problem:
\begin{subequations}\label{eq:genPDE}
\begin{align}
&\partial_t\langle\mu_{(t,\mathbf x,\bm\xi)},\id_{\mathbb R}\rangle+\dive_{\mathbf x}\langle\mu_{(t,\mathbf x,\bm\xi)},\mathbf f(\cdot,\bm \xi) \rangle=0,\quad(t,\mathbf x,\bm\xi)\in\mathbb R_+\times\mathbb R^n\times\mathbf\Xi,\label{subeq:conLaw}\\
&\mu_{(0,\mathbf x,\bm\xi)}=\sigma_0,\quad(\mathbf x,\bm\xi)\in\mathbb R^n\times\mathbf\Xi,
\end{align}
\end{subequations}
where $\langle\cdot,\cdot\rangle$ denotes the integration of a (vector-valued) function $g\in\mathcal C(\mathbb R ; \mathbb R^k)$ against a measure $\mu\in\mathcal M(\mathbb R)$, defined by
\begin{equation*}
\langle\mu,g\rangle:=\int_{\mathbb R}g(y)\mu(dy) \in \mathbb{R}^k,
\end{equation*}
while $\sigma_0=\delta_{u_0}$.
Equation \eqref{subeq:conLaw} has to be understood in a weak entropy sense, as explained in the following.
\begin{definition}[Parametric entropy measure-valued (MV) solution]
Consider a family of entropy pairs $\mathcal E$.
Let $\sigma_0$ be Young measure on $\mathbb R^n\times\mathbf\Xi$, and let $\mathbf f$ satisfying Assumption \ref{ass:assOnf}.
A Young measure $\mu$ is a parametric entropy MV solution to \eqref{eq:genPDE} for $\mathcal E$, if, for a family of entropy pairs $\mathcal E$, for all $(\eta,\mathbf q)\in\mathcal E$ and all non-negative test functions $\phi\in\mathcal C(\mathbf\Xi;\mathcal C^1_c(\mathbb R_+\times\mathbb R^n))$, it satisfies
\begin{multline}
\int_{\mathbb R_+}\int_{\mathbb R^n}\int_{\mathbf\Xi}\left(\partial_t\phi(t,\mathbf x,\bm\xi)\langle\mu_{(t,\mathbf x,\bm\xi)},\eta_v\rangle+\nabla_{\mathbf x}\phi(t,\mathbf x,\bm\xi)\cdot\langle\mu_{(t,\mathbf x,\bm\xi)},\mathbf q_v\rangle\right)d\rho(\bm\xi)d\mathbf xdt\\+
\int_{\mathbb R^n}\int_{\mathbf\Xi}\phi(0,\mathbf x,\bm\xi)\langle\sigma_0,\eta_v\rangle d\rho(\bm\xi)d\mathbf x\geq 0.
\end{multline}
\end{definition}

With the injection
\begin{align*}
    \mathcal L^\infty(\mathbb R_+\times\mathbb R^n\times\mathbf\Xi)&\rightarrow\left(\mathbb R_+\times\mathbb R^n\times\mathbf\Xi\rightarrow\mathcal P(\mathbb R)\right)\\
    u&\mapsto\left((t,\mathbf x,\bm\xi)\mapsto\delta_{u(t,\mathbf x,\bm\xi)}\right),
\end{align*}
we notice that, under the condition that $\sigma_0=\delta_{u_0}$, weak-parametric entropy solutions are parametric entropy MV solutions, 
but without further assumptions, parametric entropy MV solutions are not necessarily weak-parametric entropy solutions. However, the following result shows that the parametric entropy MV solution can be concentrated on the graph of the weak-parametric entropy solution.
\begin{theorem}\label{th:entMvSolCon}
    Let $u_0\in\mathcal L^\infty(\mathbb R^n \times \mathbf\Xi)$ and $\mathbf f$ satisfy Assumption \ref{ass:assOnf}. Let $u$ be the unique weak parametric entropy solution for $\mathcal E_K$ and $\mu$ be a parametric entropy MV solution for $\mathcal E_K$. If $\rho$-almost everywhere $\sigma_0=\delta_{u_0(\cdot)}$, then $\rho$-almost everywhere $\mu=\delta_{u(\cdot)}$.
\end{theorem}
\begin{proof}
    First, we may note that from the same arguments that those presented in the second part of the proof of Theorem \ref{th:forEqu}, a parametric entropy MV solution $\mu$ for $\mathcal E_K$ verifies the following inequality, for all $v\in\mathbb R$, all non-negative test functions $\psi\in\mathcal C^1_c(\mathbb R_+\times\mathbb R^n)$ and $\rho$-almost all $\bm\xi$:
    \begin{multline}
        \int_{\mathbb R_+}\int_{\mathbb R^n}\left(\partial_t\psi(t,\mathbf x)\langle\mu_{(t,\mathbf x,\bm\xi)},\eta_v\rangle+\nabla_{\mathbf x}\psi(t,\mathbf x)\cdot\langle\mu_{(t,\mathbf x,\bm\xi)},\mathbf q_v\rangle\right)d\mathbf xdt\\+
        \int_{\mathbb R^n}\psi(0,\mathbf x)\langle\sigma_0,\eta_v\rangle d\mathbf x\geq 0.
    \end{multline}
    Then, for $\rho$-almost all $\bm\xi$, $\mu(\cdot,\cdot,\bm\xi)$ is an entropy MV solution of the initial problem with a fixed parameter $\bm\xi$, that is a parameter-independent problem studied in \cite{marx1}. Then, from \cite[Theorem 1]{marx1} and \cite{diperna85}, we have that $\rho$-almost everywhere, if $(\sigma_0)_{(\cdot,\bm\xi)}=\delta_{u(\cdot,\bm\xi)}$, then $\mu_{(\cdot,\cdot,\bm\xi)}=\delta_{u(\cdot,\cdot,\bm\xi)}$, with $u(\cdot,\cdot,\bm \xi)$ the weak parametric entropy solution. 
\end{proof}

\subsection{Restrictions to compact hypercubes}

In order to extend the strategy developed in~\cite{marx1}, it is mandatory to work on compact sets. Whereas introducing compact domains in time and in the parameter set is trivial, the restriction to bounded space domains has to be carefully done. To simplify the setting and to avoid the problem of introducing boundary conditions to conservation laws, see for instance \cite{bardos_first_1979,otto_initial-boundary_1996,necas_weak_1996}, we assume
that the solution has no interaction with its boundary, i.e. the solution is known on the boundary of the spatial domain at any time. Let
\begin{equation}
\mathbf T:=[0,T],\quad\mathbf X:=[L_1,R_1]\times\cdots\times[L_n,R_n],\quad \mathbf \Xi:=[0,1]^p
\end{equation}
be the respective domains of time $t$, space variable $\mathbf x$ and parameter $\bm\xi$ for fixed (but arbitrary) constants $T$, $(L_i)_{i=1}^n$ and $(R_i)_{i=1}^n$. The absence of interaction with the boundary is translated as follows: 
initial data $u_0$ that we consider are the restrictions to $\mathbf X\times\mathbf\Xi$ of initial data defined on $\mathbb R^n\times\mathbf\Xi$ such that, considering the associated weak parametric entropy solution $u$, there exists $\epsilon>0$ such that in $\mathbf \partial\bm X_\epsilon:=(\partial\mathbf X+B(0,\epsilon))\cap\mathbf X$,  and for all $t\in\mathbf T$ and $\rho$-almost all $\bm\xi, u(t,\cdot,\bm\xi)=u_0(\cdot,\bm\xi)$, i.e. the weak parametric entropy solution is stationary   on $\partial\bm X_\epsilon$. This framework is the one we shall use in the following.

From \eqref{eq:staPro}, we can consider that $u$ takes values in the following compact set
\begin{equation}\label{eq:U}
\mathbf U:=[\underline u,\overline u],
\end{equation}
where the bounds are $\underline u:=\essinf_{\mathbf X,\mathbf\Xi}u_0$ and $\overline u:=\esssup_{\mathbf X,\mathbf\Xi}u_0$.

This leads us to reformulate the problem on the restricted domain.
\begin{proposition}[Parametric entropy measure-valued solution on compact hypercubes]
Consider a family of entropy pair $\mathcal E$.
Let $\mu:(t,\mathbf x,\bm\xi)\in\mathbb R_+\times\mathbb R^n\times\mathbf\Xi\mapsto\mu_{(t,\mathbf x,\bm\xi)}\in\mathcal M_+(\mathbf U)$ be a parametric entropy measure-valued solution for $\mathcal E$. Then it satisfies for all $(\eta,\mathbf q)\in\mathcal E$ and for all non-negative test functions $\phi\in\mathcal C(\mathbf\Xi;\mathcal C^1(\mathbf T\times\mathbf X))$,
\begin{multline}\label{eq:mvEntCom}
\int_{\mathbf T}\int_{\mathbf X}\int_{\mathbf\Xi}\left(\partial_t\phi(t,\mathbf x,\bm\xi)\langle\mu_{(t,\mathbf x,\bm\xi)},\eta\rangle+\nabla_{\mathbf x}\phi(t,\mathbf x,\bm\xi)\cdot\langle\mu_{(t,\mathbf x,\bm\xi)},\mathbf q\rangle\right)d\rho(\bm\xi)d\mathbf xdt\\
+\int_{\mathbf X}\int_{\mathbf\Xi}\phi(0,\mathbf x,\bm\xi)\langle\sigma_0,\eta\rangle d\rho(\bm\xi)d\mathbf x-\int_{\mathbf X}\int_{\mathbf\Xi}\phi(T,\mathbf x,\bm\xi)\langle\sigma_T,\eta\rangle d\rho(\bm\xi)d\mathbf x\\
-\int_{\mathbf T}\int_{\mathbf X} \int_{\mathbf\Xi} \phi(t,\reviewde{\mathbf x},\bm\xi) \langle \bm \gamma , \mathbf q(\cdot,\bm \xi) \rangle 
d\rho(\bm\xi)d\mathbf x dt\geq 0
\end{multline}
where $\sigma_0$ and 
 $\sigma_T$ are Young measures supported on $\mathbf X\times\mathbf\Xi$, and where $\bm \gamma$ is such that
\begin{align*}&\int_{\mathbf T}\int_{\mathbf X} \int_{\mathbf\Xi} \phi(t,\reviewde{\mathbf x},\bm\xi) \langle \bm \gamma , \mathbf q(\cdot,\bm \xi) \rangle d\rho(\bm\xi) \reviewun{d\mathbf x}dt= \\
&\sum_{i=1}^n\int_{\mathbf T}\int_{\reviewde{\mathbf X}} \int_{\mathbf\Xi}\phi(t,\mathbf x_{L,i},\bm\xi)\langle\gamma_{\reviewde{L},i},q_i(\cdot,\bm\xi)\rangle d\rho(\bm\xi) \reviewde{\delta_{\Gamma_{L,i}}(d\mathbf x)} dt 
 \\
&-\sum_{i=1}^n \int_{\mathbf T}\int_{\reviewde{\mathbf X}} \int_{\mathbf\Xi}\phi(t,\mathbf x_{R,i},\bm\xi)\langle\gamma_{\reviewde{R},i},q_i(\cdot,\bm\xi)\rangle d\rho(\bm\xi) \reviewde{\delta_{\Gamma_{R,i}}(d\mathbf x)} dt,
\end{align*} 
 where for each $1\le i\le n$, $\gamma_{L,i}$ and $\gamma_{R,i}$ are boundary measures supported on $\mathbf T\times \Gamma_{L,i}\times \mathbf\Xi$ and $\mathbf T\times \Gamma_{R,i}\times \mathbf\Xi$ respectively, with  $\Gamma_{L,i} = \{ \mathbf x \in \partial \mathbf X : x_i = L_i\}$ and $\Gamma_{R,i} = \{ \mathbf x \in \partial\mathbf X : x_i = R_i\}$,  $\mathbf x_{B,i}$ denotes the vector $(x_1,\dots,x_{i-1},\Gamma_{B,i},x_{i+1},\dots,x_n)$ for $B\in\{L,R\}$.
\end{proposition}

\begin{lemma}\label{lem:mvCom}
    Consider a family of entropy pairs $\mathcal E$ such that either $(\id,\mathbf f)\in\mathcal E$ or $\mathcal E=\mathcal E_K$. Let $\mu$ be a parametric entropy MV solution for $\mathcal E$. Then for all test functions $\phi\in\mathcal C(\mathbf\Xi;\mathcal C^1(\mathbf T\times\mathbf X))$, it satisfies
    \begin{multline}\label{eq:mvCom}
        \int_{\mathbf T}\int_{\mathbf X}\int_{\mathbf\Xi}\left(\partial_t\phi(t,\mathbf x,\bm\xi)\langle\mu_{(t,\mathbf x,\bm\xi)},\id\rangle+\nabla_{\mathbf x}\phi(t,\mathbf x,\bm\xi)\langle\mu_{(t,\mathbf x,\bm\xi)},\mathbf f\rangle\right)d\rho(\bm\xi)d\mathbf xdt\\
        +\int_{\mathbf X}\int_{\mathbf\Xi}\phi(0,\mathbf x,\bm\xi)\langle\sigma_0,\id\rangle d\rho(\bm\xi)d\mathbf x-\int_{\mathbf X}\int_{\mathbf\Xi}\phi(T,\mathbf x,\bm\xi)\langle\sigma_T,\id\rangle d\rho(\bm\xi)d\mathbf x\\
        - \int_{\mathbf T}\int_{\mathbf X} \int_{\mathbf\Xi} \phi(t,\mathbf x,\bm\xi) \langle \bm \gamma , \mathbf f(\cdot,\bm \xi) \rangle d\rho(\bm\xi) d \mathbf xdt=0.
    \end{multline}
\end{lemma}
\begin{proof}
The proof of this lemma is discussed in \cite{eymard00}, and the case of the Kruzhkov's entropies is retrieved thanks to the boundedness of $\mathbf U$. 
\end{proof}
\begin{remark}\label{rem:whyImpEquAndIne}
Our numerical method shall not use all inequalities \eqref{eq:mvEntCom}. Therefore, imposing \eqref{eq:mvCom} as an additional constraint may be beneficial in practice, see further discussion in Section \ref{sec:GMPAndRel}.
\end{remark}

\begin{remark}[Imposing constraints on the boundary]\label{rem:bouCon}
To ensure concentration of $\mu_{(t,x,\bm\xi)}$, in addition to the condition $\sigma_0=\delta_{u_0(\cdot)}$, one may impose conditions on the boundary measures $(\gamma_{L,i})_{i=1}^n$ and $(\gamma_{R,i})_{i=1}^n$.  The choice of boundary condition allows to ensure the absence of interaction with the boundary. We shall make the assumption that the trace of $u_0$ on $\Gamma_{B,i}$, noted $\gamma_{B,i}(u_0)$ exists for all $B\in\lbrace L,R\rbrace$ and all $1\leq i\leq n$, and we at the same time notice that this trace does not depend on $\bm\xi\in\mathbf\Xi$. We then want to impose that $\gamma_{B,i}(\mathbf x)=\delta_{\gamma_{B,i}(u_0)(\mathbf x)}$ for almost all $\mathbf x\in\mathbf X$, for all $1\leq i\leq n$ and for $B\in\lbrace L,R\rbrace$.
\end{remark}
\reviewun{\begin{remark}[General Dirichlet boundary conditions]
        The first entropy weak formulation with Dirichlet boundary conditions has been proposed by \cite{BLN}, but it requires the existence of strong traces, so that it cannot be generalized to MV solutions. Later, Otto introduced boundary entropy-flux pairs in his PhD thesis. A short presentation is given in \cite{otto} while an extended analysis is available in \cite{necas_weak_1996}. This notion allowed Vovelle in \cite{vovelle} to show that the Otto's framework can be reformulated via Kruzhkov semi-entropies and to extend  the wellposedness theory of conservation laws on bounded domains to entropy process solutions, which is a similar concept to MV solutions (see also \cite{panov} for more results). Lastly, in order to express such a formulation as moment contraints, it suffices to follow the approach of doubling the number of measures in \cite{marx1}, recalled in Appendix \ref{ap:defSplMea} for the case of Kruzhkov entropies.
    \end{remark}}    
Let $\nu\in\mathcal M(\mathbf K)_+$ with $\mathbf K:=\mathbf T\times\mathbf X\times\mathbf\Xi\times\mathbf U$ defined by
\begin{equation}\label{eq:dnu}
d\nu(t,\mathbf x,\bm\xi,y)=dtd\mathbf xd\rho(\bm\xi)\mu_{(t,\mathbf x,\bm\xi)}(dy)
\end{equation}
where $\mu$ is a parametric entropy MV solution. The measure $\nu$, called occupation measure (see \cite{lasserre08}), has $\lambda_{\mathbf T}\otimes\lambda_{\mathbf X}\otimes\rho$ for marginal in $(t,\mathbf x,\bm \xi)$, and $\mu_{(t,\mathbf x,\bm\xi)}$ as the conditional measure in $y$ given $(t,\mathbf x,\bm\xi)$. In the case where $\mu_{(t,\mathbf x,\bm\xi)}(dy)=\delta_{u(t,\mathbf x,\bm\xi)}(dy)$, $\nu$ is supported on the graph of the function $u$. We also introduce the time boundary measures
\begin{equation}
d\nu_0(t,\mathbf x,\bm\xi,y):=\delta_0(dt)d\mathbf xd\rho(\bm\xi)\sigma_0(dy),\quad d\nu_T(t,\mathbf x,\bm\xi,y):=\delta_T(dt)d\mathbf xd\rho(\bm\xi)\sigma_T(dy)
\end{equation}
whose supports are $\mathbf K_0:=\lbrace0\rbrace\times\mathbf X\times\mathbf\Xi\times\mathbf U$ and $\mathbf K_T:=\lbrace T\rbrace\times\mathbf X\times\mathbf\Xi\times\mathbf U$ respectively. Similarly, we introduce the space boundary measures
\begin{align}
d\nu_{L,i}(t,\mathbf x,\bm\xi,y):=dt\delta_{\Gamma_{L,i}}(d\mathbf x)d\rho(\bm\xi)\gamma_{L,i}(dy),\\d\nu_{R,i}(t,\mathbf x,\bm\xi,y):=dt\delta_{\Gamma_{R,i}}(d\mathbf x)d\rho(\bm\xi)\gamma_{R,i}(dy)
\end{align}
whose supports are given by $\mathbf K_{L,i}:=\mathbf T\times \Gamma_{L,i}\times\mathbf\Xi\times\mathbf U$ and $\mathbf K_{R,i}:=\mathbf T\times\Gamma_{R,i}\times\mathbf\Xi\times\mathbf U$ respectively, for $1 \le i \le n$. For conciseness, we shall define the  collection of measures $\bm\nu:=\{\nu,\nu_0,\nu_T,(\nu_{L,i})_{i=1}^n,(\nu_{R,i})_{i=1}^n\}$.

It is known that measures with compact support are fully characterized by their moments, see e.g. \cite[p.52]{lasserre09}. Thus, marginal constraints on occupation measures $\bm \nu$ will be imposed through their moments, see details in Appendix \ref{ap:impConMomNu}.

The introduction of the measure $\nu$ allows us to rewrite the constraints \eqref{eq:mvCom} and \eqref{eq:mvEntCom}. The equation \eqref{eq:mvCom} can be put in the following form
\begin{multline}\label{eq:F}
F(\phi,\bm\nu):=\int_{\mathbf K}\left(\partial_t\phi(t,\mathbf x,\bm\xi)y+\nabla_{\mathbf x}\phi(t,\mathbf x,\bm\xi)\cdot\mathbf f(y)\right)d\nu(t,\mathbf x,\bm\xi,y)\\
+\int_{\mathbf K}\phi(t,\mathbf x,\bm\xi)yd\nu_0(t,\mathbf x,\bm\xi,y)-\int_{\mathbf K}\phi(t,\mathbf x,\bm\xi)yd\nu_T(t,\mathbf x,\bm\xi,y)\\
+\sum_{i=1}^n\left(\int_{\mathbf K}\phi(t,\mathbf x,\bm\xi)f_i(y)d\nu_{L,i}(t,\mathbf x,\bm\xi,y)-\int_{\mathbf K}\phi(t,\mathbf x,\bm\xi)f_i(y)d\nu_{R,i}(t,\mathbf x,\bm\xi,y)\right)=0,
\end{multline}
where $\phi\in\mathcal C(\mathbf\Xi;\mathcal C^1(\mathbf T\times\mathbf X))$, and the equation \eqref{eq:mvEntCom} can be written  
\begin{multline}\label{eq:G}
G(\phi,\bm\nu,\eta,\bm q):=\int_{\mathbf K}\left(\partial_t\phi(t,\mathbf x,\bm\xi)\eta(y)+\nabla_{\mathbf x}\phi(t,\mathbf x,\bm\xi)\cdot\mathbf q(y)\right)d\nu(t,\mathbf x,\bm\xi,y)\\
+\int_{\mathbf K}\phi(t,\mathbf x,\bm\xi)\eta(y)d\nu_0(t,\mathbf x,\bm\xi,y)-\int_{\mathbf K}\phi(t,\mathbf x,\bm\xi)\eta(y)d\nu_T(t,\mathbf x,\bm\xi,y)\\
+\sum_{i=1}^n\left(\int_{\mathbf K}\phi(t,\mathbf x,\bm\xi)q_i(y)d\nu_{L,i}(t,\mathbf x,\bm\xi,y)-\int_{\mathbf K}\phi(t,\mathbf x,\bm\xi)q_i(y)d\nu_{R,i}(t,\mathbf x,\bm\xi,y)\right)\geq0,
\end{multline}
for  entropy pairs $(\eta,\mathbf q)$ in a family $\mathcal E$ and $\phi\in\mathcal C(\mathbf\Xi;\mathcal C^1(\mathbf T\times\mathbf X))$  non-negative test functions.

\reviewde{
    \begin{remark}[On systems of conservation laws]
        When studying systems of conservation laws, the main difficulty is the lack of uniqueness theory. In particular, even if entropy inequalities are available, there is no hope to identify entropy measure-valued solutions with entropy weak solutions. Nevertheless, our method could be applied as such to systems of conservation laws. 
    \end{remark}
}

\section{Moment-SOS method for measure-valued solutions on compact sets}\label{sec:GMPAndRel}

In the previous section, we introduced parametric measure-valued (MV) solutions for scalar hyperbolic equations, that are defined by equations \eqref{eq:F}-\eqref{eq:G}. The aim of this section is to express these equations as constraints on the moments of the occupation measure and to explain how to approximate these moments based on the moment-SOS (Lasserre's) hierarchy \cite{lasserre09}. For that, we require the assumption that $\mathbf f : \mathbb{R} \times \mathbf \Xi$ is a polynomial function.

We will see in the next section how to extract from these moments some information on the solution $u$ of the initial problem.



\subsection{From weak formulations to moment constraints}\label{subsec:froWeaFor}
The following lemma, derived from \cite[Lemma 1]{marx1}, relies on density arguments, together with the fact that we are working with compact sets.  
\begin{lemma}
Let $\left\lbrace\phi^{\bm\alpha}\right\rbrace_{\bm\alpha\in\mathbb N^{n+p+1}}$ be a polynomial basis on $\mathbf T\times\mathbf X\times\mathbf\Xi$. Then equation \eqref{eq:F} is equivalent to
\begin{equation}\label{eq:couConLaw}
F(\phi^{\bm\alpha},\bm\nu)=0
\end{equation}
for all ${\bm\alpha}\in\mathbb N^{n+p+1}$, where $F$ is defined in \eqref{eq:F}.
\end{lemma}
For $\mathbf f$ a polynomial function, \eqref{eq:couConLaw} provides constraints  on linear combinations of moments of measures $\bm\nu$.  
In the case of a family of polynomial entropy pairs  $\mathcal E \subseteq \mathcal E_P$, we can also express \eqref{eq:G} as constraints on the moments of measures $\bm\nu$.
\begin{lemma}\label{lem:conEntLaw}
    Assume $\left\lbrace\phi^{\bm\alpha}\right\rbrace_{\bm\alpha\in\mathcal F}$ for $\mathcal F\subset\mathbb N^{n+p+1}$ is a countable family of polynomials on $\mathbf T\times\mathbf X\times\mathbf\Xi$ such that any non-negative polynomial can be decomposed on this family, with positive coefficients. Then, equation \eqref{eq:G} is equivalent to
        \begin{equation}\label{eq:conEntLaw}
            G(\phi^{\bm\alpha},\bm\nu,\eta,\mathbf q)\geq0
        \end{equation}
    for all ${\bm\alpha}\in\mathcal F$ and all $(\eta,\mathbf q) \in \mathcal E$.
\end{lemma}

\begin{remark}\label{rem:ineq-imply-eq-poly}
    Since, as stated in Lemma~\ref{lem:mvCom}, equation \eqref{eq:G} implies equation \eqref{eq:F}, then equation \eqref{eq:conEntLaw} implies equation \eqref{eq:couConLaw} with an appropriate family of entropy pairs. It thus may seem redundant to enforce both, but{, in the approximation method, the family of polynomials in Lemma \ref{lem:conEntLaw} will be reduced, so that this} implication is no more guaranteed and imposing \eqref{eq:couConLaw} as additional constraints may be beneficial. \reviewde{An illustration is provided in Appendix \ref{ap:no_dyn} for the numerical example studied in Section \ref{subsec:RieProPos}.}
\end{remark}

The case $\mathcal E=\mathcal E_K$ ensures concentration of the measure, as seen in Theorem \ref{th:entMvSolCon}, but we are faced with two issues: first, taking into account an uncountable family of functions parametrised by $v\in\mathbf U$ and, second, the absolute value function $v\mapsto\vert v\vert$ is not a polynomial. To deal with the uncountable family of functions, we introduce $v$ as a new variable. To treat the absolute value, we double the number of measures.

More precisely, we introduce as new unknowns Borel measures $\vartheta^+$ and $\vartheta^-$, 
whose supports are respectively defined by
\begin{equation*}
\supp(\vartheta^+) = \mathbf K^+: =\lbrace(t,\mathbf x,\bm\xi,y,v)\in\mathbf K\times\mathbf U:y\geq v\rbrace,
\end{equation*}
\begin{equation*}
\supp(\vartheta^-)= \mathbf K^- :=\lbrace(t,\mathbf x,\bm\xi,y,v)\in\mathbf K\times\mathbf U:y\leq v\rbrace,
\end{equation*}
and impose the condition that $\nu \otimes \lambda_{\mathbf U} = \vartheta^+ + \vartheta^-$, which can be expressed as constraints between moments of $\nu$, $\vartheta^+$ and $\vartheta^-.$
Similarly, we introduce time boundary measures $\vartheta_0^+$, $\vartheta_0^-$, $\vartheta_T^+$ and $\vartheta_T^-$, space boundary measures $(\vartheta_{L,i}^+)_{i=1}^n$, $(\vartheta_{L,i}^-)_{i=1}^n$, $(\vartheta_{R,i}^+)_{i=1}^n$ and $(\vartheta_{R,i}^-)_{i=1}^n$, and the corresponding constraints with measures $\bm \nu$. All those definitions are plainly written in Appendix \ref{ap:defSplMea}.
We shall once again introduce a collection of measures
\begin{equation*}
    \bm\vartheta:=(\vartheta^+,\vartheta^-,\vartheta_0^+,\vartheta_0^-,\vartheta_T^+,\vartheta_T^-,(\vartheta_{L,i}^+)_{i=1}^n,(\vartheta_{L,i}^-)_{i=1}^n,(\vartheta_{R,i}^+)_{i=1}^n,(\vartheta_{R,i}^-)_{i=1}^n).
\end{equation*}
From \cite[Lemma 2]{marx1}, equation \eqref{eq:G} is equivalent to
\begin{multline}\label{eq:HPsi2}
H(\phi,\bm\vartheta):=\int_{\mathbf K}\theta(v)\left(\partial_t\phi(t,\mathbf x,\bm\xi)(y-v)+\nabla_{\mathbf x}\phi(t,\mathbf x,\bm\xi)\cdot(\mathbf f(y)-\mathbf f(v))\right)d\vartheta^+\\
+\int_{\mathbf K}\theta(v)\left(\partial_t\phi(t,\mathbf x,\bm\xi)(v-y)
+\nabla_{\mathbf x}\phi(t,\mathbf x,\bm\xi)\cdot(\mathbf f(v)-\mathbf f(y))\right)d\vartheta^-\\
+\int_{\mathbf K}\theta(v)\phi(t,\mathbf x,\bm\xi)(y-v)d\vartheta_0^++\int_{\mathbf K}\theta(v)\phi(t,\mathbf x,\bm\xi)(v-y)d\vartheta_0^-\\
-\int_{\mathbf K}\theta(v)\phi(t,\mathbf x,\bm\xi)(y-v)d\vartheta_T^+-\int_{\mathbf K}\theta(v)\phi(t,\mathbf x,\bm\xi)(v-y)d\vartheta_T^-\\
+\sum_{i=1}^n\left(\int_{\mathbf K}\theta(v)\phi(t,\mathbf x,\bm\xi)(f_i(y)-f_i(v))d\vartheta_L^++\int_{\mathbf K}\theta(v)\phi(t,\mathbf x,\bm\xi)(f_i(v)-f_i(y))d\vartheta_L^-\right. \\
\left. -\int_{\mathbf K}\theta(v)\phi(t,\mathbf x,\bm\xi)(f_i(y)-f_i(v))d\vartheta_R^+-\int_{\mathbf K}\theta(v)\phi(t,\mathbf x,\bm\xi)(f_i(v)-f_i(y))d\vartheta_R^-\right)\geq0,
\end{multline}
for all non-negative test functions $\phi\in \mathcal C^1(\mathbf T\times\mathbf X)\otimes \mathcal C(\mathbf\Xi)$ and all non-negative test functions $\theta\in\mathcal C(\mathbf U)$.

\begin{lemma}
Assume $\left\lbrace\phi^{\bm\alpha}\right\rbrace_{\bm\alpha\in \mathcal F}$ is a countable family of polynomials on $\mathbf T\times\mathbf X\times\mathbf\Xi\times\mathbf U$ such that any non-negative polynomial can be decomposed on this family with positive coefficients. 
Then  \eqref{eq:HPsi2} is equivalent to
\begin{equation}\label{eq:HPsiAlpha}
H(\phi^{\bm\alpha},\bm\vartheta)\geq0
\end{equation}
for all ${\bm\alpha\in \mathcal F}$.
\end{lemma} 
\begin{proof}
The proof relies on density arguments.
\end{proof}

A particular family $\left\lbrace\phi^{\bm\alpha}\right\rbrace_{\bm\alpha\in\mathbb N^{2(n+p+2)}}$ satisfying the assumption that any non-negative polynomial can be decomposed on this family with positive coefficients is given by
\begin{multline*}
\phi^{\bm\alpha}(t,\mathbf x,\bm\xi,v):=t^{\alpha_1}(T-t)^{\alpha_2}\prod_{i=1}^n\left((x_i-L_i)^{\alpha_{2i+1}}(R_i-x_i)^{\alpha_{2(i+1)}}\right)\\
\prod_{i=1}^p\left((\xi_i)^{\alpha_{2i+2n+1}}(1-\xi_i)^{\alpha_{2(i+n+1)}}\right)(v-\underline u)^{\alpha_{2n+2p+3}}(\overline u-v)^{\alpha_{2(n+p+2)}}
\end{multline*}
for $\bm\alpha\in\mathbb N^{2(n+p+2)}$.
The proof follows the one of the Lemma 3 in \cite{marx1}, and uses Handelman's Positivstellensatz \cite{marx1}.

\subsection{Generalized Moment Problem}

Roughly speaking, the Generalized Moment Problem (GMP) is an infinite-dimensional linear optimization problem on finitely many Borel measures $\nu_i\in\mathcal M(\mathbf K_i)_+$, with $\mathbf K_i\subseteq\mathbb R^{n_i}$, with $i=1,...,N$ and $n_i\in\mathbb N$. That is, one is interested in finding measures whose moments satisfy (possibly countably many) linear constraints and which minimize some criterion. In full generality, the GMP is intractable, but if all $\mathbf K_i$ are basic semi-algebraic sets\footnote{A basic semi-algebraic set is defined by $\lbrace\mathbf x\in\mathbb R^n :  f_i(\mathbf x)\geq0,\forall i=1,\dots,m\rbrace$ where $m\in\mathbb N$ and $f_1,\dots,f_m$ are polynomials.} and the integrands are polynomials, then one may provide an efficient numerical scheme to approximate as closely as desired any finite number of moments of optimal solutions of the GMP. It consists of solving a hierarchy of semi-definite programs\footnote{A semidefinite program is a particular class of a convex conic optimization problem that can be solved numerically efficiently.} of increasing size. Convergence of this numerical scheme is guaranteed by invoking powerful results from real algebraic geometry, essentially positivity certificates, and further developed for many classical cases in \cite{tacchi21, korda22}.

Let $h_i\in\mathbb R[\mathbf w^i]$ and $h_{i,k}\in\mathbb R[\mathbf w^i]$ be polynomials in the vector of indeterminates $\mathbf w^i\in\mathbb R^{n_i}$ and let $b_k$ be real numbers, for finitely many $i=1,\ldots,N$ and countably many $k=1,2,\ldots$. The GMP is the following optimization problem over measures:
\begin{equation}\label{eq:GMP}
\begin{aligned}
\inf_{\nu_1,\ldots,\nu_N}&\sum_{i=1}^N\int_{\mathbf K_i}h_id\nu_i =:\rho^*\\
\text{s.t.}&\sum_{i=1}^N\int_{\mathbf K_i}h_{i,k}d\nu_i\leq b_k,\quad k=1,2,\ldots\\
&\nu_i\in\mathcal M(\mathbf K_i)_+,\quad i=1,\ldots,N.
\end{aligned}
\end{equation}
\subsection{From measures to moments and their approximation}
Instead of optimizing over the measures in problem \eqref{eq:GMP}, we optimize  over their moments. For simplicity and clarity of exposition, we describe the approach in the case of a single unknown measure $\nu$, but it easily extends to the case of several measures. Let us consider the simplified GMP
\begin{equation}\label{eq:simGMP}
\begin{aligned}
\inf_\nu&\int_{\mathbf K}hd\nu:=\rho^*\\
\text{s.t.}&\int_{\mathbf K}h_kd\nu\leq b_k,\quad k=1,2,\ldots\\
&\nu\in\mathcal M(\mathbf K)_+,
\end{aligned}
\end{equation}
where $\mathbf{K}$ is a compact set in $\mathbb{R}^n$, $h\in\mathbb R[\mathbf w]$, $h_k\in\mathbb R[\mathbf w]$ and $b_k\in\mathbb R$ for all $k=1,2,\dots$.
The moment sequence $\mathbf z=(\mathbf z_{\bm\alpha})_{\bm\alpha\in\mathbb N^n}$ of a measure $\nu\in\mathcal M(\mathbf K)_+$ is defined by
\begin{equation}\label{eq:momSeq}
\mathbf z_{\bm\alpha}=\int_{\mathbf K}\mathbf w^{\bm\alpha} d\nu,\quad\bm\alpha\in\mathbb N^n.
\end{equation}
Similarly, given a sequence $\mathbf z=(\mathbf z_{\bm\alpha})_{\bm\alpha\in\mathbb N^n}$, if \eqref{eq:momSeq} holds for some $\nu\in\mathcal M(\mathbf K)_+$ we say that the sequence has the representing measure $\nu$. Recall that measures on compact sets are uniquely characterized by their moments (see \cite[p. 52]{lasserre09}).
\reviewde{
\begin{remark}
The use of canonical moments, i.e. associated with monomials, may be critical from a numerical point of view. Other polynomial bases with more favorable numerical properties could be considered. However, in the numerical experiments, we restrain ourselves to domains included in the unit hypercube, which moderates numerical instabilities.  
\end{remark}
}
Let $\mathbb N_d^n:=\lbrace\bm\alpha\in\mathbb N^n:\vert\bm\alpha\vert\leq d\rbrace$, where $\vert\bm\alpha\vert:=\sum_{i=1}^n\alpha_i$, and $n_d:=\binom{n+d} d$. A vector $\mathbf p:=(\mathbf p_{\bm\alpha})_{\bm\alpha\in\mathbb N_d^n}\in\mathbb R^{n_d}$ is the coefficient vector (in the monomial basis) of a polynomial $p\in\mathbb R[\mathbf w]$ with degree $d=\deg(p)$ expressed as $p=\sum_{\bm\alpha\in\mathbb N_{d}^n}\mathbf p_{\bm\alpha}\mathbf w^{\bm\alpha}$. Integrating $p$ with respect to a  measure $\nu$ involves only finitely many moments:
\begin{equation*}
\int_{\mathbf K}pd\nu=\int_{\mathbf K}\sum_{\bm\alpha\in\mathbb N_d^n}\mathbf p_{\bm\alpha}\mathbf w^{\bm\alpha} d\nu=\sum_{\bm\alpha\in\mathbb N_d^n}\mathbf p_{\bm\alpha}\int_{\mathbf K}\mathbf w^{\bm\alpha} d\nu=\sum_{\bm\alpha\in\mathbb N_d^n}\mathbf p_{\bm\alpha}\mathbf z_{\bm\alpha}.
\end{equation*}
Next, we define a pseudo-integration with respect to an arbitrary sequence $\mathbf z\in\mathbb R^{\mathbb N^n}$ by
\begin{equation}
\ell_{\mathbf z}(p):=\sum_{\bm\alpha\in\mathbb N^n}\mathbf p_{\bm\alpha}\mathbf z_{\bm\alpha}
\end{equation}
and $\ell_{\mathbf z}$ is called the Riesz functional. 
\begin{theorem}[Riesz-Haviland {\cite[Theorem 3.1]{lasserre09}}]\label{th:rieHav}
Let $\mathbf K\subseteq\mathbb R^n$ be closed. A real sequence $\mathbf z\in\mathbb R^{\mathbb N^{{n}}}$ is the moment sequence of some measure $\nu\in\mathcal M(\mathbf K)_+$, i.e. $\mathbf z$ satisfies~\eqref{eq:momSeq}, if and only if $\ell_{\mathbf z}(p)\geq 0$ for all $p\in\mathbb R[\mathbf w]$ non-negative on $\mathbf K$.
\end{theorem}
Assuming that $\mathbf K$ is closed, we can reformulate thanks to this result the GMP \eqref{eq:simGMP} as a linear problem on moment sequences, namely
\begin{equation}\label{eq:momGMP}
\begin{aligned}
\inf_{\mathbf z}\quad&\ell_{\mathbf z}(h)=\rho^*\\
\text{s.t.}\quad&\ell_{\mathbf z}(h_k)\leq b_k,\quad k=1,2,\ldots\\
&\ell_{\mathbf z}(p)\geq 0,\text{ for all }p\in\mathbb R[\mathbf w]\text{ non-negative on }\mathbf K.
\end{aligned}
\end{equation}
Theorem \ref{th:rieHav} guarantees the equivalence between formulations \eqref{eq:momGMP} and \eqref{eq:simGMP}. However, the latter reformulation is still numerically intractable.

\paragraph{From non-negative polynomials to sums of squares}
Characterizing non-negativity of polynomials is an important issue in real algebraic geometry. Let $\mathbf K$ be a basic semi-algebraic set, i.e.
\begin{equation}\label{eq:K}
\mathbf K=\lbrace\mathbf w\in\mathbb R^n:g_1(\mathbf w)\geq 0,\ldots,g_m(\mathbf w)\geq 0\rbrace
\end{equation}
for some polynomials $g_1,\ldots,g_m\in\mathbb R[\mathbf w]$, and assume that $\mathbf K$ is compact. In addition assume that one of the polynomials, say the first one, is $g_1(\mathbf w):=N-\sum_{i=1}^n  w_i^2$ for some $N$ sufficiently large\footnote{This condition is slightly stronger than asking $\mathbf K$ to be a basic semi-algebraic compact set. However, the inequality $N-\sum_{i=1}^n w_i^2\geq 0$ can always be added as a redundant constraint to the description of a basic semi-algebraic \textit{compact} set. This condition has to be added because Putinar's result applies to a family of polynomials, and is not inherent to the set this family describes.}. For notational convenience we let $g_0(\mathbf w):=1$.

We say that a polynomial $s\in\mathbb R[\mathbf w]$ is a sum of squares (SOS) if there are finitely many polynomials $q_1,\ldots,q_r$ such that $s(\mathbf w)=\sum_{j=1}^rq_j(\mathbf w)^2$ for all $\mathbf w$.
\begin{theorem}[Putinar’s Positivstellensatz]\label{th:putPos}
If $p>0$ on the basic semi-algebraic compact set $\mathbf K$ defined by \eqref{eq:K} with $g_1(\mathbf w):=N-\sum_{i=1}^n  w_i^2$, then $p=\sum_{j=0}^ms_jg_j$ for some SOS polynomials $s_j\in\mathbb R[\mathbf w],j=0,1,\ldots,m$.
\end{theorem}
By a density argument, checking non-negativity of $\ell_{\mathbf z}$ on polynomials non-negative on $\mathbf K$ can be replaced by checking non-negativity only on polynomials that are strictly positive on $\mathbf K$ and hence on those that have a SOS representation as in Theorem \ref{th:putPos}.

For a given integer $d$, denote by $\Sigma[\mathbf w]_d\subset\mathbb R[\mathbf w]$ the set of SOS polynomials of degree at most $2d$, and define the cone $Q_d(g)\subset\mathbb R[\mathbf w]$ for $g=(g_0,\ldots,g_m)$ by
\begin{equation}
Q_d(g):=\left\lbrace\sum_{j=0}^m\sigma_jg_j:\deg(\sigma_jg_j)\leq 2d,\sigma_j\in\Sigma[\mathbf w]_d,j=0,1,\ldots,m\right\rbrace
\end{equation}
and observe that $Q_d(g)$ consists of polynomials which are non-negative on $\mathbf K$.

Let $\mathbf{b}_d(\mathbf w):=(\mathbf w^{\bm\alpha})_{\vert\bm\alpha\vert\leq d}\in\mathbb R[\mathbf w]^{n_d}$ be the vector of monomials of degree at most $d$. We recall that $n_d$ denotes the binomial number $\binom{n+d}{n}$. For $j=0,...,m$, let $d_j=\lceil \deg(g_j)/2\rceil$, let $M_{d-d_j}(g_j\mathbf{z})$ denote the real symmetric matrix linear in $\mathbf{z}$ corresponding to the entrywise application of $\ell_{\mathbf{z}}$ to the matrix with polynomial entries $g_j\mathbf{b}_{d-d_j}(\mathbf w)\mathbf{b}_{d-d_j}^T(\mathbf  w)$. For $j=0$ and $g_0=1$, the matrix $M_{d}(\mathbf{z}) = \ell_{\mathbf z}(\mathbf{b}_d \mathbf{b}_d^T)$ (where $\ell_{\mathbf z}$ is applied entrywise) is called the moment matrix. For any other value of $j$, it is called a localizing matrix. It turns out that, for all $j=0,1,\ldots,m$, $\ell_{\mathbf z}(g_jq^2)\geq 0$ for all $q\in\mathbb R[\mathbf w]_d$ if and only if $M_{d-d_j}(g_j\mathbf{z})\succeq 0$, which are convex linear matrix inequalities in $\mathbf z$ and where $\succeq$ denotes the positive semi-definite (or Loewner) order.

\paragraph{Moment-SOS hierarchy}
The following finite-dimensional semi-definite programming (SDP) problems are relaxations of the moment problem \eqref{eq:momGMP}:
\begin{equation}\label{eq:momSOSHier}
\begin{aligned}
\inf_{\mathbf z\in\mathbb R^{n_{2d}}}\quad&\ell_{\mathbf z}(h):=\rho^*_d\\
\text{s.t.}\quad&\ell_{\mathbf z}(h_k)\leq b_k,\deg(h_k)\leq 2d,k=1,2,\ldots\\
&M_{d-d_j}(g_j\textbf{z})\succeq 0,j=0,1,\ldots,m
\end{aligned}
\end{equation}
and they are parametrized by the relaxation degree $d\geq\max_{j=0,\ldots,m}d_j$. 
\begin{theorem}[Convergence of the moment-SOS hierarchy, {\cite[Theorem \reviewde{7}]{tacchi21}}]\label{th:LasHieCon}
Suppose that\reviewde{$\ \mathbf K$ is a basic semi-algebraic compact set. Further assume that} there exists $C>0$ such that for any $d\in \mathbb{N}$, if $\mathbf z^{d}\in\mathbb R^{n_{2d}}$ is solution of \eqref{eq:momSOSHier}, then $z^d_{\mathbf 0}\leq C$, with $C$ independent of $d$. \reviewde{Finally assume that there exists a unique solution $\nu^*\in\mathcal M(\mathbf K)_+$ to problem \eqref{eq:simGMP}}. Then there exists a sequence $(\mathbf z^d)_{d}=((\mathbf z^d_{\bm\alpha})_{\bm\alpha\in\mathbb N^{n_{2d}}})_{d}$ such that $\ell_{\mathbf z^d}(h)=\rho_d^*$ and for all $\bm\alpha\in\mathbb N^{n}$
\begin{equation}
\mathbf z_{\bm\alpha}^d\underset{d\rightarrow\infty}\longrightarrow\int_{\mathbf K}\mathbf w^{\bm\alpha}d\nu(\mathbf w).
\end{equation}
In particular, one has $\rho_d^*\rightarrow\rho^*$ as $d\rightarrow\infty$.
\end{theorem}

\subsection{Application to our problem}\label{subsec:appPro}
\paragraph{Entropy MV solution as a GMP}
In the scalar hyperbolic case, the measures $\nu_i$ under consideration are from the collection $\bm\nu$, or $\bm\nu$ and $\bm\vartheta$ when considering Kruzhkov's entropies. The sets $\mathbf K_i$ all correspond to $\mathbf K=\mathbf T\times\mathbf X\times\mathbf\Xi\times\mathbf U$. The polynomials $h_{i,j}$ are given in \eqref{eq:couConLaw} (conservation law), \eqref{eq:conEntLaw} when considering polynomial entropy pairs or  \eqref{eq:HPsiAlpha} (and compatibility conditions between $\bm\nu$ and $\bm\vartheta$ \eqref{eq:comConTheNu} and similar equations) when considering Kruzhkov entropy pairs (entropy inequalities), and \eqref{eq:1bouMea}-\eqref{eq:lasMarCon} (marginal constraints). For the sake of readibility, we shall only consider the case of polynomial entropies and a formulation only on measures  $\bm\nu$.

We may also define an objective functional
\begin{multline}\label{eq:objFun}
\int_{\mathbf K}hd\nu+\int_{\mathbf K}h_0d\nu_0+\int_{\mathbf K}h_Td\nu_T+\sum_{i=1}^n\left(\int_{\mathbf K}h_{L,i}d\nu_{L,i}+\int_{\mathbf K}h_{R,i}d\nu_{R,i}\right),
\end{multline}
with $h$, $h_0$, $h_T$, $(h_{L,i})_{i=1}^n$, $(h_{R,i})_{i=1}^n\in\mathbb R[t,\mathbf x,\bm\xi,y]$.

If the initial measure is concentrated on the graph of the initial condition and if, in addition, one imposes suitable boundary measures as exposed in Remark \ref{rem:bouCon}, then the choice of the objective functional is not crucial to recover the entropy MV solution of scalar hyperbolic PDE. Indeed, as a consequence of Theorem \ref{th:entMvSolCon}, the corresponding Young measure is concentrated: there is nothing to be optimized. However, our aim is to approximate the GMP by a finite dimensional optimization problem in order to solve it numerically and, then, the choice of the objective functional will impact the convergence of the corresponding relaxations. From experimental observations, two objective functionals seem to produce interesting results: the maximum of the opposite of the entropy constraints and the minimum of the trace of moment matrix. Choosing the latter seems to be a good heuristic: minimizing the nuclear norm of a matrix leads to reducing its rank (see \cite{recht10}), which tends to favorise measures with localized support. However, there is up to date still no proof of a general effective functional.

Finally, one is able to define a GMP:
\begin{equation}\label{eq:entMvSolGMP}
\begin{aligned}
&\inf_{\nu,\nu_T}&&\eqref{eq:objFun}\text{ (objective functional)}\\
&\text{s.t.}&&\eqref{eq:couConLaw}\text{ (conservation law)},\\
&&&\eqref{eq:conEntLaw}\text{ (entropy inequality)},\\
&&&\eqref{eq:1bouMea}-\eqref{eq:lasMarCon}\text{ (marginal constraints)},
\end{aligned}
\end{equation}
where the infimum is taken over measures $\nu\in\mathcal M(\mathbf K)_+,\nu_T\in\mathcal M(\mathbf K_T)_+$.

\reviewde{
    Theorem \ref{th:LasHieCon} extends to the case of multiple measures, as discussed in \cite{lasserre09} and shown in \cite{tacchi21}. Note that the compact sets $\mathbf T$, $\mathbf X$, $\mathbf\Xi$ and $\mathbf U$ as defined before can be expressed as basic semi-algebraic compact sets, so that $\mathbf K$ is also a basic semi-algebraic compact set:
    \begin{equation}
        \begin{aligned}
            &\mathbf T=\lbrace t\in\mathbb R:t(T-t)\geq 0\rbrace,\\
            &\mathbf X=\lbrace\mathbf x\in\mathbb R:(x_1-L_1)(R_1-x_1)\geq 0,\dots,(x_n-L_n)(R_n-x_n)\geq 0\rbrace,\\
            &\mathbf\Xi=\lbrace\bm\xi\in\mathbf\Xi:\xi_1(1-\xi_1)\geq 0,\ldots,\xi_p(1-\xi_p)\geq 0\rbrace,\\
            &\mathbf U=\lbrace y\in\mathbb R:(y-\underline u)(\overline u-y)\geq 0\rbrace.
            \end{aligned}
    \end{equation}
    Moreover, the constraint for $\bm\alpha=\mathbf 0$ (see equation \eqref{eq:1bouMea} in Appendix \ref{ap:impConMomNu}) yields the relaxed linear constraint $\mathbf z_{\mathbf 0}=\ell_{\mathbf z}(1)=\int_{\mathbf T\times\mathbf X\times\mathbf\Xi}dtdxd\rho(\bm\xi)\leq\vert\mathbf T\vert\vert\mathbf X\vert$. Finally, supposing that $\sigma_0$ is concentrated on the graph of the weak-parametric entropy solution, \eqref{eq:entMvSolGMP} admits a unique solution. Hence the hypotheses of Theorem \ref{th:LasHieCon} are satisfied.
}

Then, optimal solutions of the moment-SOS hierarchy \eqref{eq:momSOSHier} (adapted to the present context) converge to optimal solutions of \eqref{eq:entMvSolGMP} as $d$ goes to infinity. In particular, one may extract the MV solution of \eqref{eq:mvCom}, provided that $\sigma_0$, $\gamma_{L,i}$ and $\gamma_{R,j}$ are concentrated for $1\leq i\leq n$, as already discussed in Remark \ref{rem:bouCon}.

\subsection{Post-processing quantities of interest}
We have seen in the previous section how to obtain approximate sequences $\mathbf z^d$ of  moments of the measure $\nu$ on $\mathbf K$, such that $d\nu(t,\mathbf x , \bm \xi , y) = d\mu_{t,\mathbf x , \bm \xi}(y) dt d\mathbf x d\rho(\bm\xi)$ where $\mu$ is the measure-valued solution supported on the graph of the solution. In this section, we present how to construct an approximation of the function $u$ thanks to the Christoffel-Darboux function and its ability to estimate the support of a measure (see \cite{lasserre22} for further details). Also, we show how to obtain approximations of statistical moments of variables of interest that are functions of the solution, possibly using a moment completion technique and the Christoffel-Darboux function. 

\subsubsection{Approximation of the graph of the solution}\label{subsubsec:appGraSol}

We consider that we have obtained an approximation $\mathbf z^d$ of the moments of order $2d$ of the measure  $\nu$, which is a measure supported on the graph of the function $u(t,\mathbf{x},\bm \xi)$. In order to approximate the function from the moments, we rely on an approximate Christoffel-Darboux function associated with the measure (that has to be carefully defined), which tends to take high values on the support of the measure. Thus, finding the minimizers of the approximate inverse Christoffel-Darboux function for given $(t,\mathbf x,\bm\xi)\in\mathbf T\times\mathbf X\times\mathbf\Xi$ gives an approximation of $u(t,\mathbf x,\bm\xi)$. 
For $\mathbf{w} = (t,\mathbf x,\bm \xi, y) \in \mathbf K$, we  let $\mathbf{b}_d(\mathbf{w})$ be a basis of monomials of order up to $d$ and $M_d(\mathbf{z}^d)=\ell_{\mathbf z^d}(\mathbf{b}_d(\cdot)\mathbf{b}_d(\cdot)^T) $ be the corresponding moment matrix, that is the Gram matrix of the basis $\mathbf{b}_d(\mathbf{w})$ for the measure $\nu^d$ corresponding to $\mathbf{z}^d$. When $M_d(\mathbf{z}^d)$ is invertible, the inverse Christoffel-Darboux function is defined by
$$
q_{\nu^d}(\mathbf w) = \mathbf{b}_d(\mathbf{w})^T M_d(\mathbf{z}^d)^{-1} \mathbf{b}_d(\mathbf{w}) = \sum_{i=1}^{n_d} \lambda_i^{-1} (\mathbf{b}_d(\mathbf{w})^T\mathbf v_i)^2 
$$
where the $(\lambda_i,\mathbf{v}_i)$ are eigenpairs of  $M_d(\mathbf{z}^d)$, and 
the polynomials $p_i(\mathbf{w}) = \lambda_i^{-1/2} \mathbf{b}_d(\mathbf{w})^T\mathbf v_i$ form an  orthonormal basis of the space of polynomials of order $d$ in $L^2_{\nu^d}(\mathbf{K})$.
 In the case where  $M_d(\mathbf{z}^d)$ is singular, a regularization is introduced by considering the function
 $$
 q_{\nu^d,\beta}(\mathbf w) = \mathbf{b}_d(\mathbf{w})^T (M_d(\mathbf{z}^d) + \beta \mathbf I) ^{-1} \mathbf{b}_d(\mathbf{w}) = \sum_{i=1}^{n_d} (\lambda_i + \beta)^{-1} (\mathbf{b}_d(\mathbf{w})^T\mathbf p_i)^2, 
 $$
 which turns out to be the inverse Christoffel-Darboux function of a measure $\nu^d + \beta \nu_0$, where $\nu_0$ is the measure on $\mathbf K$ for which the monomials form an orthonormal family \reviewde{(see \cite{marx2})}.
Exploiting the fact that $q_{\nu^d + \beta \nu_0}$ tends to take low values on the graph of $u$,  an approximation of $u$ is defined  by
 $$f_{\beta,d}(t,\mathbf x, \bm \xi) \in \arg\min_{y\in \mathbf U} q_{\nu^d + \beta \nu_0}(t,\mathbf x, \bm \xi,y).$$ 
Further information can be found in \cite{marx2}.

\reviewde{
\begin{remark}
In this paper, we consider  polynomial moments of the occupation measure. This allows us to exploit the localization property of the associated Christoffel functions and to estimate the graph of the solution a posteriori. The use of moments associated with other functions, such as piecewise polynomials, could probably be considered as well (see \cite{mula22} for the use of piecewise polynomial moments and a short review on the use of other types of moments).  
\end{remark}
}

\subsubsection{Statistical moments of variables of interest}\label{subsubsec:staMomVarInt}
Considering $\bm\xi$ as a random parameter, one may be interested in computing the expectation of some variable of interest $Q(\bm \xi) = F(u(\cdot,\cdot,\bm \xi) ; \bm \xi)$, where 
$F(\cdot,\bm \xi)$ is a real-valued function taking as input time-space functions.
In some particular situations, it is possible to directly obtain an estimation of this quantity from 
the moments $\mathbf z^d$.
In particular, this is the case when $$
Q(\bm \xi) = \int_{\mathbf T \times \mathbf X} G(u(t,\mathbf x , \bm \xi) , t,\mathbf x , \bm \xi) dt d\mathbf x,$$ 
with  $G$ is polynomial since then 
$$\mathbb{E}(Q(\bm \xi)) = \int_{\mathbf T \times \mathbf X \times\mathbf\Xi} G(\mathbf w) d\nu(\mathbf w) \approx \ell_{\mathbf z^d}(G).  
$$
We may also be interested in obtaining statistical moments of the solution $u(t,\mathbf x , \bm \xi)$
at different points $(t, \mathbf x)$, which is not a variable of interest in the above format. Of course, these quantities  can  be estimated from point-wise evaluations of $u$ based on the technique presented in the previous section. 
However, an alternative approach is possible to estimate  the statistical moments
\begin{equation}\label{eq:f_k}
    \int_{\mathbf\Xi}u(t,\mathbf x,\bm\xi)^kd\rho(\bm\xi) := f_k(t,\mathbf x)
\end{equation}
for all $(t,\mathbf x)\in\mathbf T\times\mathbf X$, from the the approximate moments $\mathbf z^d$ of the measure $\nu$.
We know that the measure $\nu$ can be disintegrated into its marginal $\lambda_{\mathbf T}\otimes\lambda_{\mathbf X}$ and its conditional measure $d\nu( \bm \xi ,  y\vert t,\mathbf x)$, such that $d\nu(t,\mathbf x , \bm \xi , y) = d\nu( \bm \xi ,  y\vert t,\mathbf x) dt d \mathbf x.$
\\
We assume that $f_k(t,\mathbf x)$ takes values in a compact set $\mathbf F:=[\underline F,\overline F]$ which can be easily obtained in terms of $\mathbf U$ and $k$. 
We then let $\lbrace\widetilde{g_j}\rbrace_{j=1}^m$, $m\in\mathbb N$, be polynomials that describe  the semi-algebraic compact set $\mathbf T\times\mathbf X\times\mathbf F$. Letting $\mathbf z$ be the sequence of moments of $\nu$, we may notice that for all $\bm\alpha = (\alpha_1 , \alpha_2)\in\mathbb N^{n+1}_{2d-k}$, $$\mathbf z_{\alpha_1,\alpha_2,0,k}=\int_{\mathbf T}\int_{\mathbf X}t^{\alpha_1}\mathbf x^{\alpha_2}f_k(t,\mathbf x)d\mathbf xdt =
\int_{\mathbf T}\int_{\mathbf X}\int_{\mathbf F}t^{\alpha_1}\mathbf x^{\alpha_2}y\delta_{f_k(t,\mathbf x)}(dy)d\mathbf xdt.$$

Our goal is then here to approximate the support of the measure $\delta_{f_k(t,\mathbf x)}(dy) d\mathbf x dt$ from its moments $\bm\omega$ in order to recover the graph of $f^k(t,\mathbf x)$. We are faced with the issue that the information we have on the moments is incomplete, namely, we only have the moments $\omega_{\bm\alpha,0}$ for $\bm\alpha\in\mathbb N^{n+1}_{2d}$ and $\omega_{\bm\alpha,1}$ for $\bm\alpha\in\mathbb N^{n+1}_{2d-k}$. Following \cite{henrion20}, we introduce the following finite-dimensional semi-definite programming (SDP) problems to recover the graph of $f_k(t,\mathbf x)$ from incomplete moment information:
\begin{equation}\label{optim-completion}
\begin{aligned}
&\inf_{\bm\omega\in\mathbb R^{(n+2)_d}}&&\text{Tr}(M_d(\bm\omega))\\
&\text{s.t. }&&\omega_{\bm\alpha,0}=z^d_{\alpha_1,\alpha_2,0,0},\quad\forall\bm\alpha\in\mathbb N_{2d}^{n+1}\\
&&&\omega_{\bm\alpha,1}=z^d_{\alpha_1,\alpha_2,0,k},\quad\forall\bm\alpha\in\mathbb N_{2d-k}^{n+1}\\
&&&M_{d-d_j}(\widetilde{g_j}\bm\omega)\succeq 0, \quad j=0,...,m,
\end{aligned}
\end{equation}
where $\text{Tr}(M)$ denotes the trace of a matrix $M$. We recall that $M_d(\bm\omega)$ denotes the moment matrix of $\bm\omega$.
From this, we can compute  the corresponding Christoffel-Darboux approximation of $f_k$, following the approach of the previous section, see \cite{marx2,henrion20}.

\section{Numerical examples}\label{sec:numExa}
For numerical illustration, we consider Burgers-type equations with parametrised initial condition or parametrised flux. 

The choice of entropy pairs is important to ensure uniqueness of the solution. Implementing Kruzkhov's entropy pairs is possible (as seen in Section \ref{subsec:froWeaFor}), but computationally heavy since it requires a reformulation with measures in higher dimension. It is known that the entropy $\eta(y)=y^2$ provides  sufficient constraints to ensure uniqueness of the  entropy solution for Burgers equation \cite{delellis04}. Then, instead of using  Kruzkhov's pairs, we here rely on the following family of polynomial entropies:
\begin{equation*}
    \eta_l(y)=y^{2l},\quad\forall l\in\mathbb N
\end{equation*}
and the corresponding polynomial functions $\mathbf{q}_l$. As an objective function, we choose 
the trace of the moment matrix (see discussion in section \ref{subsec:appPro}).

Numerical experiments are performed with the Matlab interface  Gloptipoly3 \cite{henrion07}. 

In order to approximate the graph of solutions $u$, we use the method described in Section \ref{subsubsec:appGraSol}. Numerically, the optimization of the Christoffel function is achieved through a discretization of $\mathbf T$, $\mathbf X$, $\mathbf\Xi$ and $\mathbf U$ and the computation of the Christoffel function at each point of the grid. \reviewde{We set the value of the regularization parameter $\beta$ to $10^{-7}$.}

\secondereview{
\begin{remark}[On the invertibility of the moment matrix]
    Imposing that the moment matrix is positive semi-definite does not guarantee its invertibility, hence the regularization. However, the regularized matrix may be  poorly conditioned for high values of relaxation order $d$, but we do not directly manipulate the ill-conditioned matrix. In order to obtain the Christoffel-Darboux kernel, we compute a spectral  decomposition of the non-regularized matrix and use the inverse of shifted eigenvalues. This seems to mitigate numerical instabilities. However, a further stability analysis should be required.
\end{remark}
}

 We shall in the following denote by $\widetilde{u_d}$ the Christoffel-Darboux approximation of the solution using  approximate moments from a degree  $d$ of the hierarchy, and by $u$ the exact solution of our Riemann problem.

\reviewun{
    \begin{remark}[On computational complexity]\label{rem:complexity}
        In our numerical experiments, we used an interior-point algorithm to solve SDP problems, whose complexity is analyzed in \cite{nesterov94}. Let $q_d$ be the number of unknown moments of a measure. Our SDP problem has  semi-positive inequalities on 
        matrices $M_{d-d_j} \in \mathbb{R}^{r_i \times r_i}$, $0\le j \le m$, with $r_j \le q_d/2$. Letting $r=\sum_{j=0}^mr_j$, the complexity of the interior-point algorithm is $O(\sqrt r(q_d^2\sum_{j=0}^mr_j+q_d\sum_{j=0}^mr_j^3))$. Thus, the complexity is $O( q_d^{9/2})$. 
    \end{remark}}
    \reviewun{
    \begin{remark}[Curse of dimension]
    For a relaxation degree $d$, the 
     number of unknown moments of a measure is $q_d=\binom{d+ n+p+2}{d} = O(d^{n+p+2})$. Thus from Remark \ref{rem:complexity}, we deduce that the complexity of the interior-point algorithm is $O(d^{\frac92(n+p+2)})$. To circumvent the curse of dimension and address problems with high dimension $n+p+2$, we should exploit low-dimensional structures in the set of moments, such as sparsity \cite{magron23} or low-rankness. This will be addressed in future works. 
    \end{remark}
    }
    \reviewun{
    \begin{remark}
        Mesh-based methods can be used to approximate solutions in terms of the physical variables and the parameters, also facing the curse of dimension. Due to the presence of discontinuities in terms of the physical variables and the parameters (not necessarily aligned with the coordinates), adaptive mesh refinement methods are required, leading to prohibitive computational cost in high dimension. Our method is mesh-free and seems to capture well discontinuities. Further comparisons need to be performed in future works.
    \end{remark}
}
\reviewde{
    \begin{remark}[Obtaining the moment constraints]
        We here consider problems with polynomial or piecewise polynomial data, that allowed us to compute exactly the moments of the given measures $\sigma_0,\gamma_{L,i}$ and $\gamma_{R,i}$, $1\leq i\leq n$. For more complicated data and high dimensional problems, numerical integration methods could be required, such as Monte Carlo methods.
    \end{remark}
}

\subsection{Riemann problem for the Burgers equation with parametrised initial condition}\label{subsec:RieProPos}
As a first example, we consider the classical one-dimensional Riemann problem (see e.g., \cite{evans97}) for a Burgers equation, with a parameter-independent flux 
\begin{equation*}
\mathbf f(u)=\frac 1 2u^2.
\end{equation*}
and where we parametrise the initial position of the shock, 
taking 
\begin{equation*}
u_0(x,\xi)=\left\lbrace
\begin{aligned}
1\text{ if }x<\frac 1 4(\xi-1),\\
0\text{ if }x\geq\frac 1 4(\xi-1).
\end{aligned}
\right.
\end{equation*}
with a  parameter $\xi$ taking values in $\mathbf\Xi=[0,1]$.
We know that  the solution takes values  in $\mathbf U=[0,1]$. The time-space window on which we consider the solution is $\mathbf T=[0,\frac 1 2]$ and $\mathbf X=[-\frac 1 2,\frac 1 2]$.


The unique  solution  is
\begin{equation}\label{eq:anaSolParIniCon}
u(t,x,\xi)=\left\lbrace
\begin{aligned}
1\text{ if }x<\frac 1 4(\xi-1)+\frac t2,\\
0\text{ if }x\geq\frac 1 4(\xi-1)+\frac t2,
\end{aligned}
\right.
\end{equation}
Equipping $\mathbf\Xi$ with the Lebesgue measure on $[0,1]$, it yields the following statistical moments
\begin{equation*}
\begin{split}
f_k(t,x) &= \int_{\mathbf\Xi}u(t,x,\xi)^k d\rho(\xi)\\
&=1-\min(1,\max(0,1-2t+4x))+0^k\min(1,\max(0,1-2t+4x)),
\end{split}
\end{equation*}
for all $k\in\mathbb N$, for all $(t,x)\in\mathbf T\times\mathbf X$. We may notice that in this simple case, $f_k$ is independent on $k$ for $k\geq1$.

\reviewde{We recall that at relaxation degree $d$, the number of unknown moments $2q_d\leq2\binom{d+ n+p+2}{d}$. Indeed, each measure is expected to yield \(q_d\) moments, and we optimize over the two measures \(\nu\) and \(\nu_T\). However, \(\nu_T\) has less unknown moments, since it is supported on \(\mathbf K_T\) and some of its moments are dependant.}
 
\paragraph{Retrieving the graph of the solution}
Figure \ref{fig:xi0} shows the graphs of the approximate solution $\widetilde{u_d}(t,x,0)$ for $(t,x)\in\mathbf T\times\mathbf X$ (so that the shock is initially located at $x=-\frac14$), with hierarchy's degree $d=2,5,8$.

\begin{figure}[H]
    \centering
    \begin{subfigure}[b]{.45\textwidth}
        \centering
        \includegraphics[width=\textwidth]{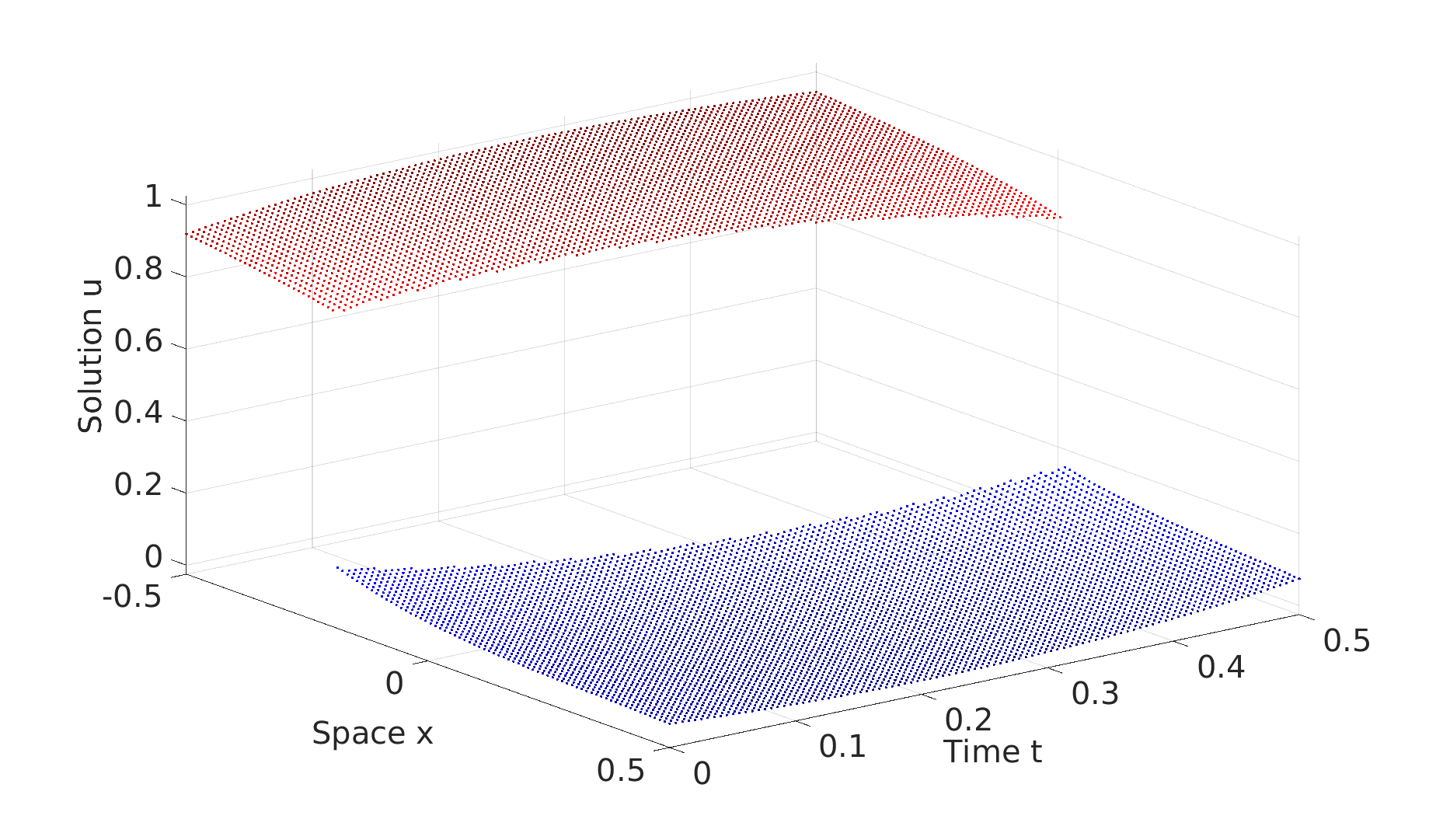}
        \caption{$d=2$}
        \label{fig:d2xi0}
    \end{subfigure}
    \begin{subfigure}[b]{0.45\textwidth}
        \centering
        \includegraphics[width=\textwidth]{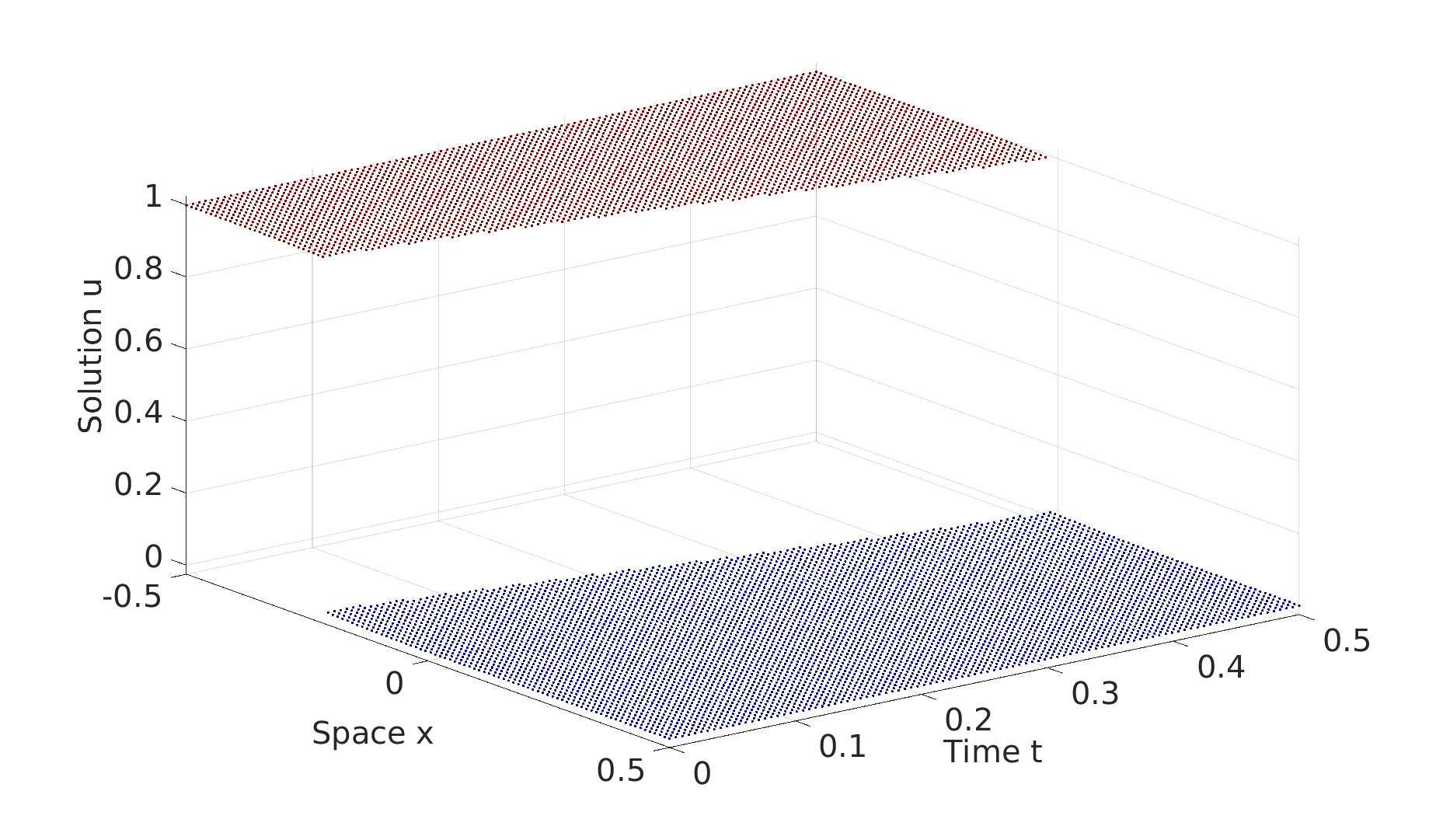}
        \caption{$d=5$}
        \label{fig:d5xi0}
    \end{subfigure}
    \begin{subfigure}[b]{0.45\textwidth}
        \centering
        \includegraphics[width=\textwidth]{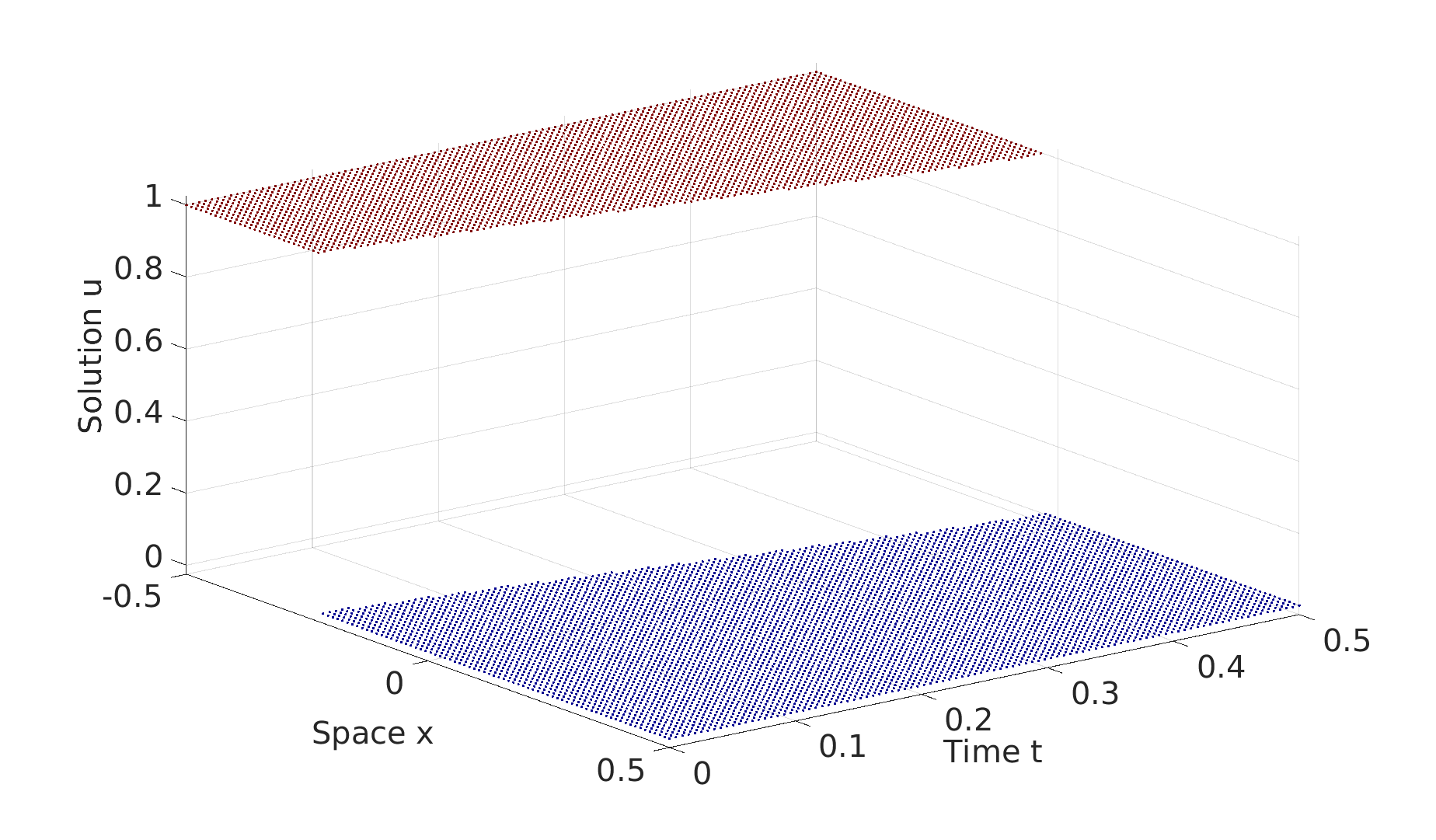}
        \caption{$d=8$}
        \label{fig:d8xi0}
    \end{subfigure}
    \caption{Graphs of the approximate solution $\widetilde{u_d}(t,x,0)$ for $d=2,5,8$}
    \label{fig:xi0}
\end{figure}

\reviewde{Figure \ref{fig:sup} shows the graphs of the approximate solution $\widetilde{u_d}(\frac14,x,\xi)$ for $x\in\mathbf X$ and \(\xi=0,1\) (so that the shock is initially located at $x=-\frac14$ and \(x=0\) respectively), with relaxation degree  $d=2,5,8$, superposed with the exact solution.}

\begin{figure}[H]
    \centering
    \begin{subfigure}[b]{.49\textwidth}
        \centering
        \includegraphics[width=\textwidth]{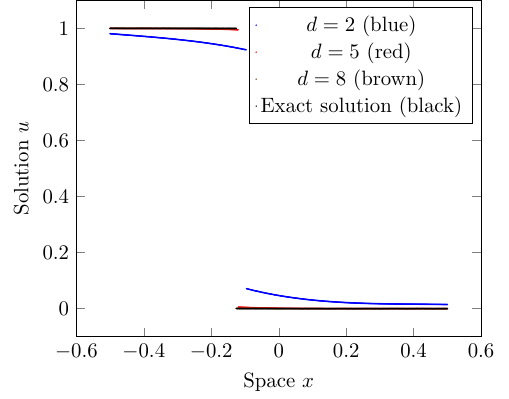}
        \caption{\reviewde{\(\xi=0\)}}
        \label{fig:supXi0}
    \end{subfigure}
    \begin{subfigure}[b]{.49\textwidth}
        \centering
        \includegraphics[width=\textwidth]{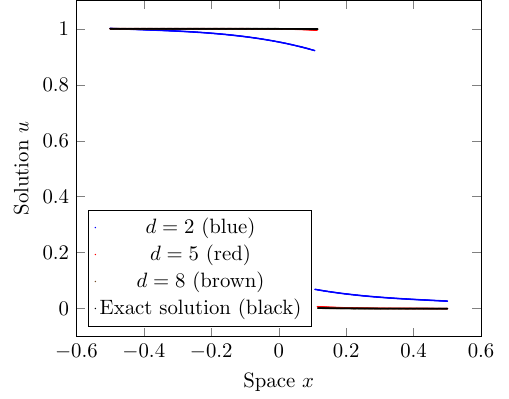}
        \caption{\reviewde{\(\xi=1\)}}
        \label{fig:supXi1}
    \end{subfigure}
    \caption{\reviewde{Graphs of the approximate solution $\widetilde{u_d}(\frac14,x,\xi)$ for $d=2,5,8$ superposed with the exact solution}}
    \label{fig:sup}
\end{figure}




\reviewde{We see that our approximation is almost indistinguishable from the exact solution for \(d=5\), and indistinguishable from the exact solution \(d=8\).} We observe the same results as in \cite{marx1}, where discontinuities are very well resolved as early as $d=5$.

\paragraph{Error estimation}
We choose to compute two different errors of our approximate solution.

\subparagraph{$l^1(\mathbf T\times\mathbf X\times\mathbf\Xi)$ error} We randomly pick $25$ values in $\mathbf\Xi$, and consider $25$ equidistant values in $\mathbf T$ and $\mathbf X$. We denote the test sets $\mathbf\Xi_e$, $\mathbf T_e$ and $\mathbf X_e$ respectively. We study the evolution of the relative $l^1$ error with respect to the degree $d$ of the hierarchy. Namely, we are interested in
\begin{equation*}
    e_g(d):=\frac{\Vert u-\widetilde{u_d}\Vert_{l^1(\mathbf T_e\times\mathbf X_e\times\mathbf\Xi_e)}}{\Vert u\Vert_{l^1(\mathbf T_e\times\mathbf X_e\times\mathbf\Xi_e)}}.
\end{equation*}
The results are presented in Table \ref{tab:errGlo} \reviewde{and on Figure \ref{fig:ploGloErr}.}
\begin{table}[h]
    \centering
    \begin{tabular}{c|c c c c c c c}
        $d$ & 2 & 3 & 4 & 5 & 6 & 7 & 8 \\
        \hline
        \reviewde{$q_d$} & \reviewde{140} & \reviewde{420} & \reviewde{990} & \reviewde{2002} & \reviewde{3640} & \reviewde{6120} & \reviewde{9690} \\
        \hline
        $e_g(d)$ & 0.0850 & 0.0267 & 0.0191 & 0.0168 & 0.0165 & 0.0167 & 0.0163
    \end{tabular}
    \caption{Number of unknowns $q_d$ and error $e_g(d)$ for $d=2,\dots,8$}
    \label{tab:errGlo}
\end{table}

\begin{figure}
    \centering
    \begin{tikzpicture}[scale = .85]
        \begin{semilogyaxis}[
            xlabel={$d$},
            ylabel={$e_g$},
        ]
        
        \addplot
            coordinates {
            (2,0.084961690837201)(3,0.026712023375538)(4,0.019084737368510)(5,0.016842460027072)(6,0.016530370254742)(7,0.016732767323110)(8,0.016263168579668)
            };
            
        \end{semilogyaxis}
    \end{tikzpicture}
    \caption{\reviewde{Evolution of the error $e_g$ with relaxation degree $d$}}
    \label{fig:ploGloErr}
\end{figure}
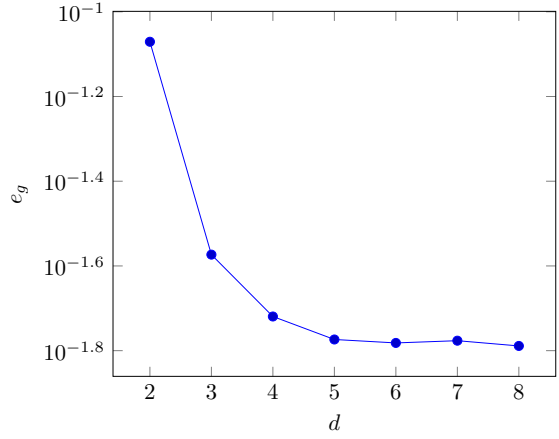

We observe a fast convergence of the error for small values of $d$. The convergence is not monotone and rather slow for high values of $d$. \reviewde{It may thus seem interesting to investigate where the errors are concentrated in our approximation. For illustrative purpose, we plot the distributed error $\varepsilon(t,x)=\vert\widetilde{u_5}(t,x,0.2)-u(t,x,0.2)\vert$ for $(t,x)\in\mathbf T_e\times\mathbf X_e$ on Figure \ref{fig:disErr}. We   observe that the errors are mostly concentrated around the shock, but it is noticeable that the closer to the time boundaries, the worse they are.}

\begin{figure}
    \centering
    \begin{subfigure}[b]{.49\textwidth}
        \centering
        \includegraphics[width=\textwidth]{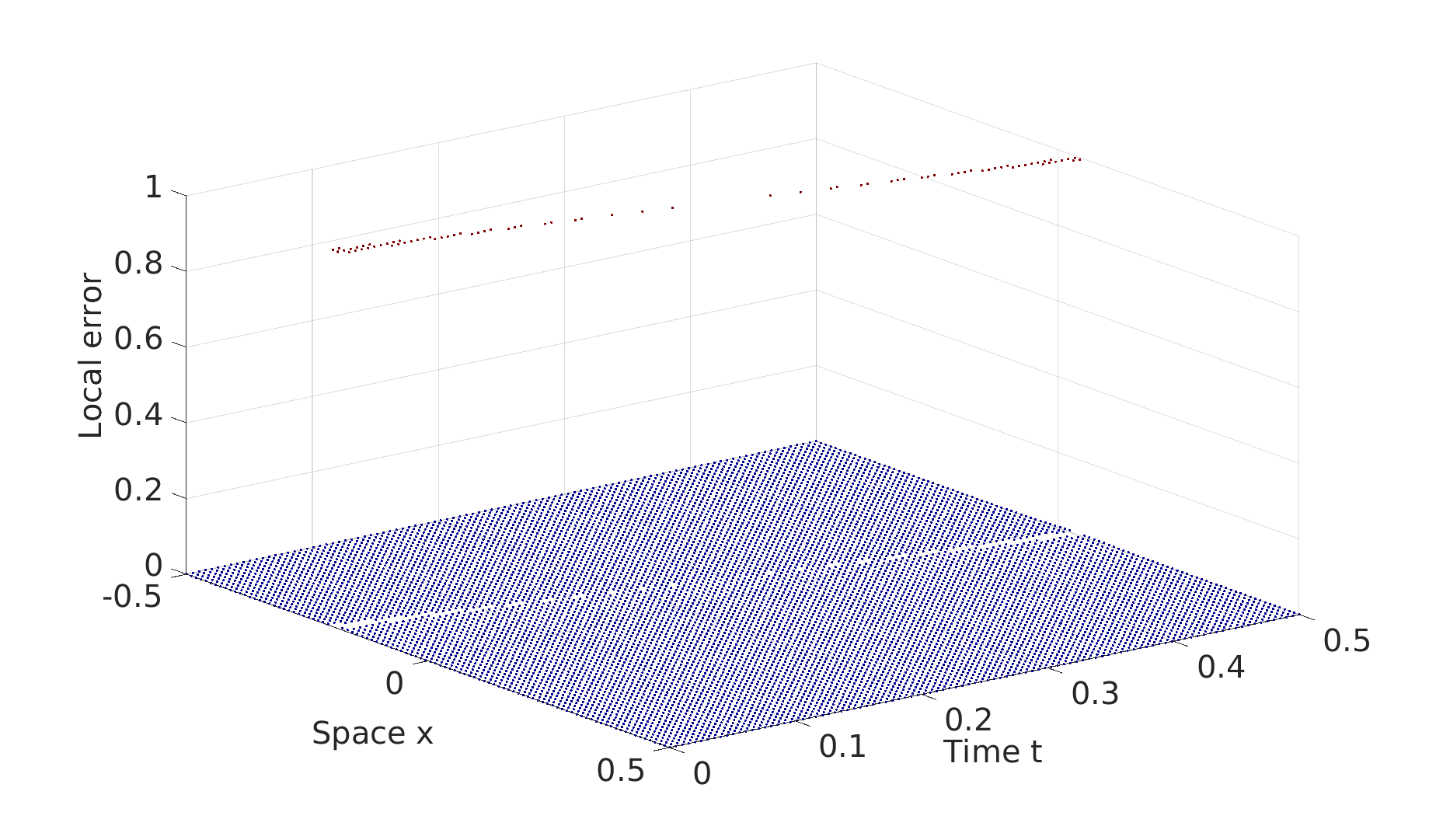}
        \caption{\reviewde{\(z\) axis from \(0\) to \(1\)}}
        \label{fig:disErrLar}
    \end{subfigure}
    \begin{subfigure}[b]{.49\textwidth}
        \centering
        \includegraphics[width=\textwidth]{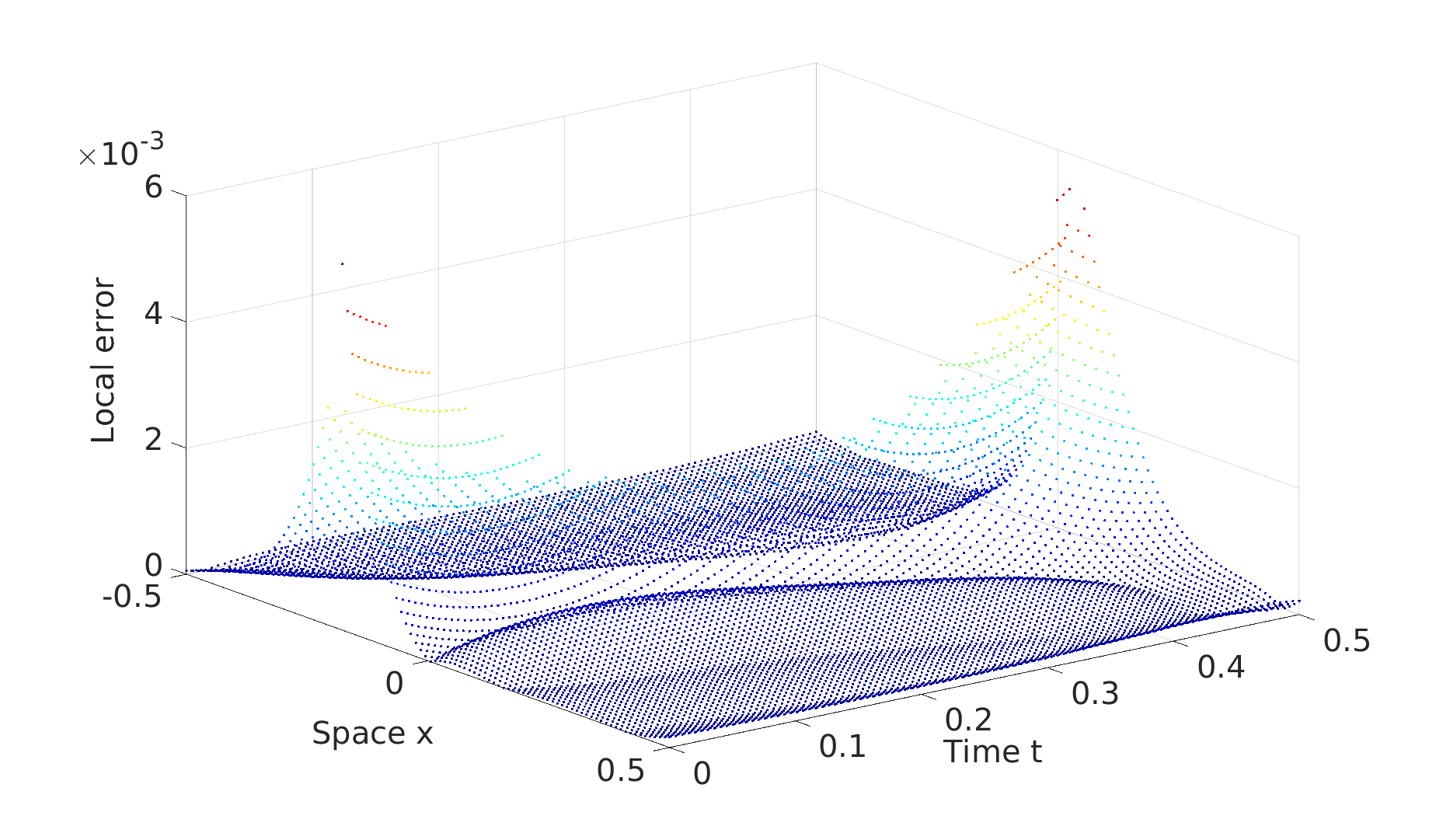}
        \caption{\reviewde{\(z\) axis from \(0\) to \(6\times10^{-3}\)}}
        \label{fig:disErrWoutSho}
    \end{subfigure}
    \caption{\reviewde{Graph of the error $\varepsilon(t,x)=\vert\widetilde{u_5}(t,x,0.2)-u(t,x,0.2)\vert$}}
    \label{fig:disErr}
\end{figure}

\reviewde{We observe similar results on other numerical tests, although for smaller values of $d$ we also observe large errors near spatial boundaries. }


\subparagraph{\secondereview{On the impact of the regularization parameter}}

\secondereview{The choice of the regularization parameter is an open question. In our experiments we chose to work with \(\beta=10^{-7}\), which avoids numerical instabilities. The results can be improved with a suitable choice of \(\beta\). For illustrative purpose, we represent the evolution of the  error \(e_g\) for different values of \(\beta\) in Table \ref{tab:beta0}.  We define \(\beta_m^d:=1.001\times\vert\lambda_m^d\vert\),  \(\lambda_{m}^d\)  the minimal eigenvalue of the moment matrix obtained for a hierarchy of degree \(d\).} \secondereview{For \(d\geq6\), \(\lambda_m^d<0\) and minimizing the Christoffel-Darboux kernel without regularization is unstable.}

\begin{table}[h]
    \centering
    \begin{tabular}{c|c|c|c|c|c|c|c}
        \(d\) & 2 & 3 & 4 & 5 & 6 & 7 & 8\\
        \hline
        \(\beta=10^{-7}\) & 0.0850 & 0.0267 & 0.0191 & 0.0168 & 0.0165 & 0.0167 & 0.0163\\
        \hline
        \(\beta=0\) & 0.0850 & 0.0264 & 0.0142 & 0.00830 & X & X & X\\
        \hline
        \(\beta=\beta_m^d\) & 0.0895 & 0.0270 & 0.0142 & 0.00856 & 0.0186 & 0.0134 & 0.00912\\
        \hline
        \(\beta_m^d\) & \(1\times10^{-4}\) & \(2\times10^{-7}\) & \(6\times10^{-10}\) & \(6\times10^{-14}\) & \(1\times10^{-12}\) & \(1\times10^{-11}\) & \(2\times10^{-11}\)
    \end{tabular}
    \caption{\secondereview{Error $e_g(d)$ for $d=2,\dots,8$ for different choices of \(\beta\)}}
    \label{tab:beta0}
\end{table}

\subparagraph{$l^1(\mathbf T\times\mathbf X)$ error for different parameter values} We consider four different values of the parameter $\xi\in \mathbf\Xi_e:=(0,0.2,0.6,1)$ (which correspond to a shock initially located at $x=-0.25$, $x=-0.2$, $x=-0.1$ and $x=0$), and $100$ equidistant points in $\mathbf T$ and $\mathbf X$, denoting the test sets $\mathbf T_e$ and $\mathbf X_e$ respectively. We then choose to study, for each $\xi_e\in\mathbf\Xi_e$, the evolution of the relative $l^1(\mathbf T_e\times\mathbf X_e)$ error with respect to the degree $d$ of the hierarchy. We are thus interested in
\begin{equation*}
    {e_p}_{\xi_e}(d):=\frac{\Vert u(\cdot,\cdot,\xi_e)-\widetilde{u_d}(\cdot,\cdot,\xi_e)\Vert_{l^1(\mathbf T_e\times\mathbf X_e)}}{\Vert u(\cdot,\cdot,\xi_e)\Vert_{l^1(\mathbf T_e\times\mathbf X_e)}},
\end{equation*}
for all $\xi_e\in\mathbf\Xi_e$.
The results are presented in Table \ref{tab:err_xi0}.
\begin{table}[h]
    \centering
    \begin{tabular}{c|c c c c c c c}
        $d$ & 2 & 3 & 4 & 5 & 6 & 7 & 8 \\
        \hline
        \reviewde{$q_d$} & \reviewde{140} & \reviewde{420} & \reviewde{990} & \reviewde{2002} & \reviewde{3640} & \reviewde{6120} & \reviewde{9690} \\
        \hline
        ${e_p}_{0}(d)$ & 0.208 & 0.0616 & 0.0343 & 0.0314 & 0.0279 & 0.0276 & 0.0271\\
        \hline
        ${e_p}_{0.2}(d)$ & 0.0971 & 0.0286 & 0.0218 & 0.0193 & 0.0176 & 0.0171 & 0.0182\\
        \hline
        ${e_p}_{0.6}(d)$ & 0.0563 & 0.0207 & 0.0162 & 0.0162 & 0.0158 & 0.0161 & 0.0174\\
        \hline
        ${e_p}_{1}(d)$ & 0.104 & 0.0407 & 0.0244 & 0.0229 & 0.0208 & 0.0194 & 0.0184
    \end{tabular}
    \caption{Number of unknowns $q_d$ and errors ${e_p}_{\xi_e}(d)$ for $d=2,\dots,8$}
    \label{tab:err_xi0}
\end{table}




We observe the same behaviour of the errors as in the previous paragraph.

\paragraph{\reviewde{Conservation condition}} \reviewde{We here check 
the conservation condition, i.e. how far $c_d(t,\xi):=\int_{\mathbf X}(\widetilde{u_d}(t,x,\xi)-u(t,x,\xi))dx$ is from $0$. We plot on Figure \ref{fig:conCon} 
 the function $t\mapsto c_d(t,0.2)$ for $11$ equidistant points in $\mathbf T$ and for $d=2,5,8$. In order to approximate the integral, we compute the pointwise error for $1001$ equidistant points in $\mathbf X$ and divide the sum by the number of points.} \reviewde{The conservation condition is rather well satisfied and it tends to improve as $d$ rises.}

\begin{figure}
    \centering
    \includegraphics[scale=0.85]{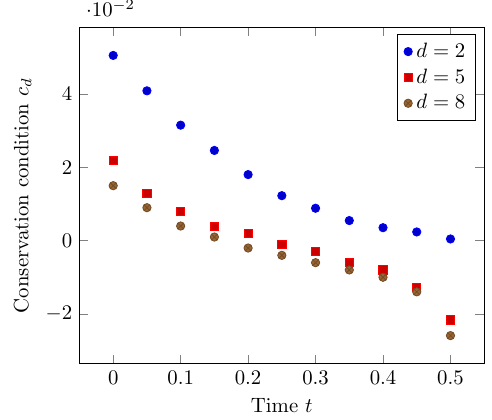}
    \caption{\reviewde{Conservation condition $c_d(t,0.2)$ versus time $t$ for $d=2,5,8$}}
    \label{fig:conCon}
\end{figure}

\paragraph{Retrieving statistical moments of the solution} Denote $\mathbf T_e$ $100$ equidistant points in $\mathbf T$ and $\mathbf X_e$ $100$ equidistant points in $\mathbf X$. We want to approximate the expectation $f_1(x,t)$  of the solution for all $(t,x)\in\mathbf T_e\times\mathbf X_e$ following the method described in Section \ref{subsubsec:staMomVarInt}. Denoting $\widetilde{{f_1}_d}$ the approximated expected value of the solution for degree of relaxation $d$ of the hierarchy, we want to compute the relative $l^1(\mathbf T_e\times\mathbf X_e)$ error of our approximation for $d=2,\dots,8$, namely, we are interested in
\begin{equation*}
    e_s(d):=\frac{\Vert f_1-\widetilde{{f_1}_d}\Vert_{l^1(\mathbf T_e\times\mathbf X_e)}}{\Vert f_1\Vert_{l^1(\mathbf T_e\times\mathbf X_e)}},
\end{equation*}
for all $d=2,\dots,8$. 
The results are presented in Table \ref{tab:err_s}.

\begin{table}[h]
    \centering
    \begin{tabular}{c|c c c c c c c}
        d & 2 & 3 & 4 & 5 & 6 & 7 & 8 \\
        \hline
        \reviewde{$q_d$} & \reviewde{140} & \reviewde{420} & \reviewde{990} & \reviewde{2002} & \reviewde{3640} & \reviewde{6120} & \reviewde{9690} \\
        \hline
        $e_s(d)$ & 0.358 & 0.102 & 0.0557 & 0.0451 & 0.0484 & 0.0574 & 0.0637
    \end{tabular}
    \caption{Number of unknowns $q_d$ and error $e_s(d)$ for $d=2,\dots,8$}
    \label{tab:err_s}
\end{table}

    We note here the same phenomenon as for the errors presented above occurring, where the approximation rapidly improves as  $d$ rises until $d=5$. The convergence is then rather slow and not monotone.

\subsection{Riemann problem for the Burgers equation with parametrised flux}

As a second illustration, we consider the classical one-dimensional Riemann problem (see e.g., \cite{evans97}) for a Burgers equation, where we parametrise the flux of the equation. In particular, we choose the flux
\begin{equation*}
f(u,\xi)=\frac14(\xi+1)u^2,
\end{equation*}
with a parameter $\xi$ taking values in $\mathbf\Xi=[0,1]$.
The Riemann problem to this conservation law is a Cauchy problem with the following initial condition, piecewise constant with one point of discontinuity:
\begin{equation*}
u_0(x)=\left\lbrace
\begin{aligned}
1\text{ if }x<0,\\
0\text{ if }x\geq0.
\end{aligned}
\right.
\end{equation*}
The solution is known to take values in $\mathbf U=[0,1]$. The time-space window on which we consider the solution is $\mathbf T=[0,\frac 1 2]$ and $\mathbf X=[-\frac 1 2,\frac 1 2]$.

The unique analytical solution corresponding to the initial condition is
\begin{equation}
u(t,x,\xi)=\left\lbrace
\begin{aligned}
1\text{ if }x<\frac14(\xi+1)t,\\
0\text{ if }x\geq\frac14(\xi+1)t,
\end{aligned}
\right.
\end{equation}
We can note that the randomness in \eqref{eq:anaSolParIniCon} was simply a translation of the solution, whereas, here, the phenomenon is non-linear, since the speed of the shock depends on $\xi$.

Providing $\mathbf\Xi$ with the Lebesgue measure on $[0,1]$, it comes that, for all $k\in\mathbb N$,
\begin{equation*}
    f_k(0,x)=u_0(x)^k,
\end{equation*}
for all $x\in\mathbf X$, and
\begin{equation*}
    f_k(t,x)=1-\min(1,\max(0,\frac{4x}t-1))+0^k\min(1,\max(0,\frac{4x}t-1)),
\end{equation*}
for all $(t,x)\in\mathbf T\times\mathbf X$. We may notice that in this simple case, for all $t>0$, $f_k(t,\cdot)$ is independent on $k$ for $k\geq1$.

\paragraph{Retrieving the graph of the solution}
Figure \ref{fig:fxi0} shows the graphs of the approximate solution $\widetilde{u_d}(t,x,0)$ for $(t,x)\in\mathbf T\times\mathbf X$ (so that the speed of the shock is $\frac14$), with relaxation degree $d=2,5,8$.

\begin{figure}[H]
    \centering
    \begin{subfigure}[b]{0.45\textwidth}
        \centering
        \includegraphics[width=\textwidth]{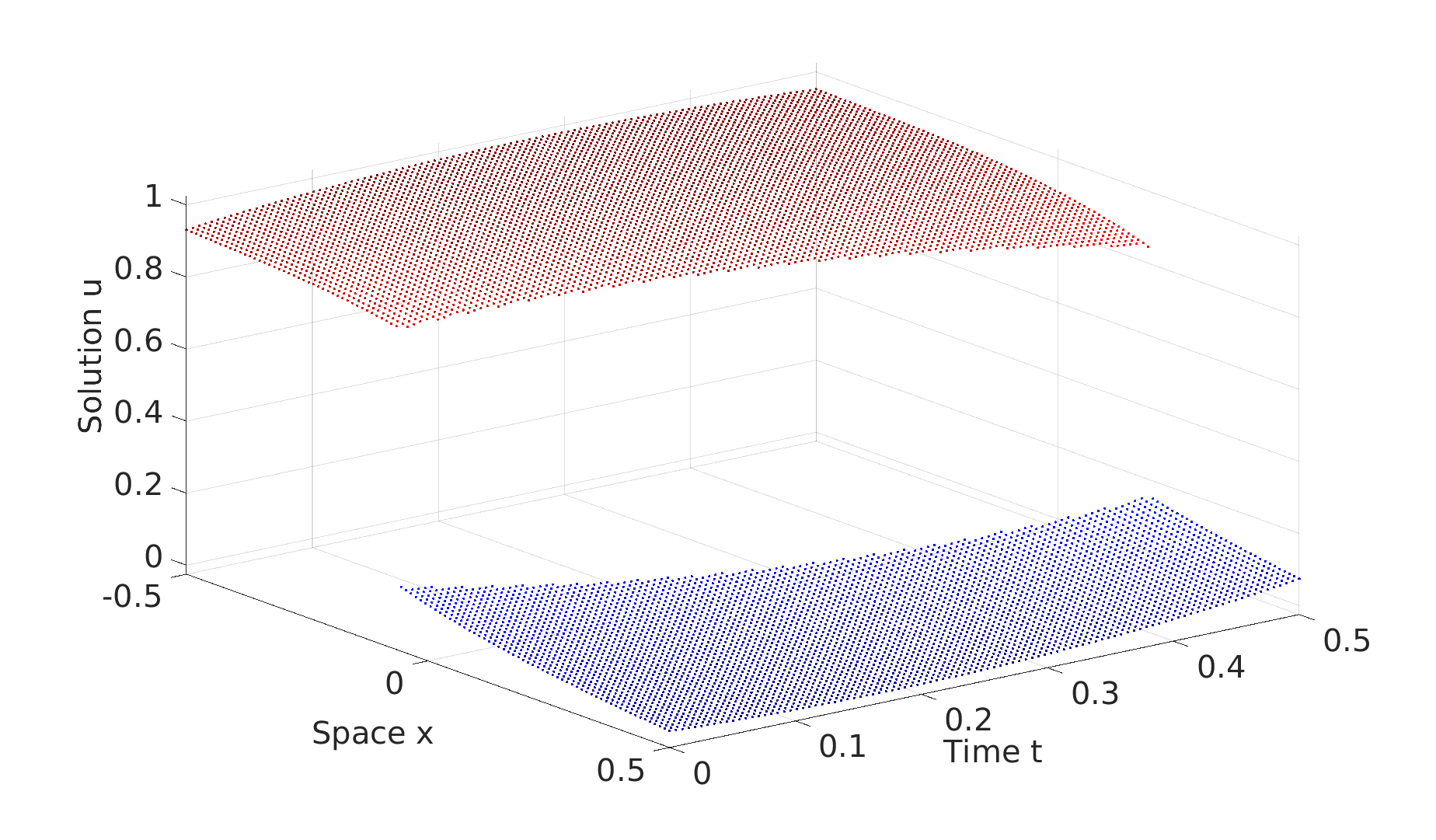}
        \caption{$d=2$}
        \label{fig:fd2xi0}
    \end{subfigure}
    \begin{subfigure}[b]{0.45\textwidth}
        \centering
        \includegraphics[width=\textwidth]{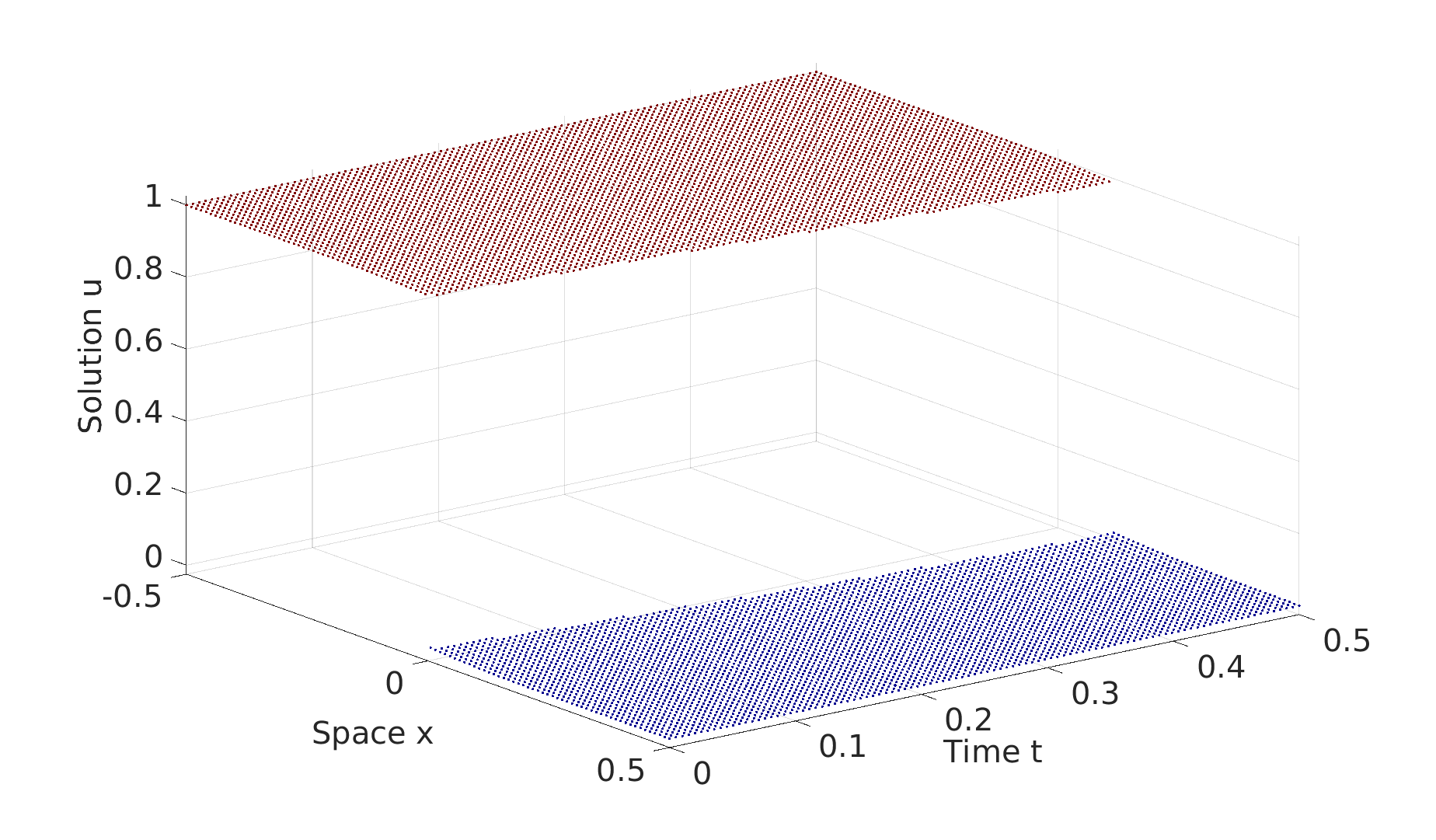}
        \caption{$d=5$}
        \label{fig:fd5xi0}
    \end{subfigure}
    \begin{subfigure}[b]{0.45\textwidth}
        \centering
        \includegraphics[width=\textwidth]{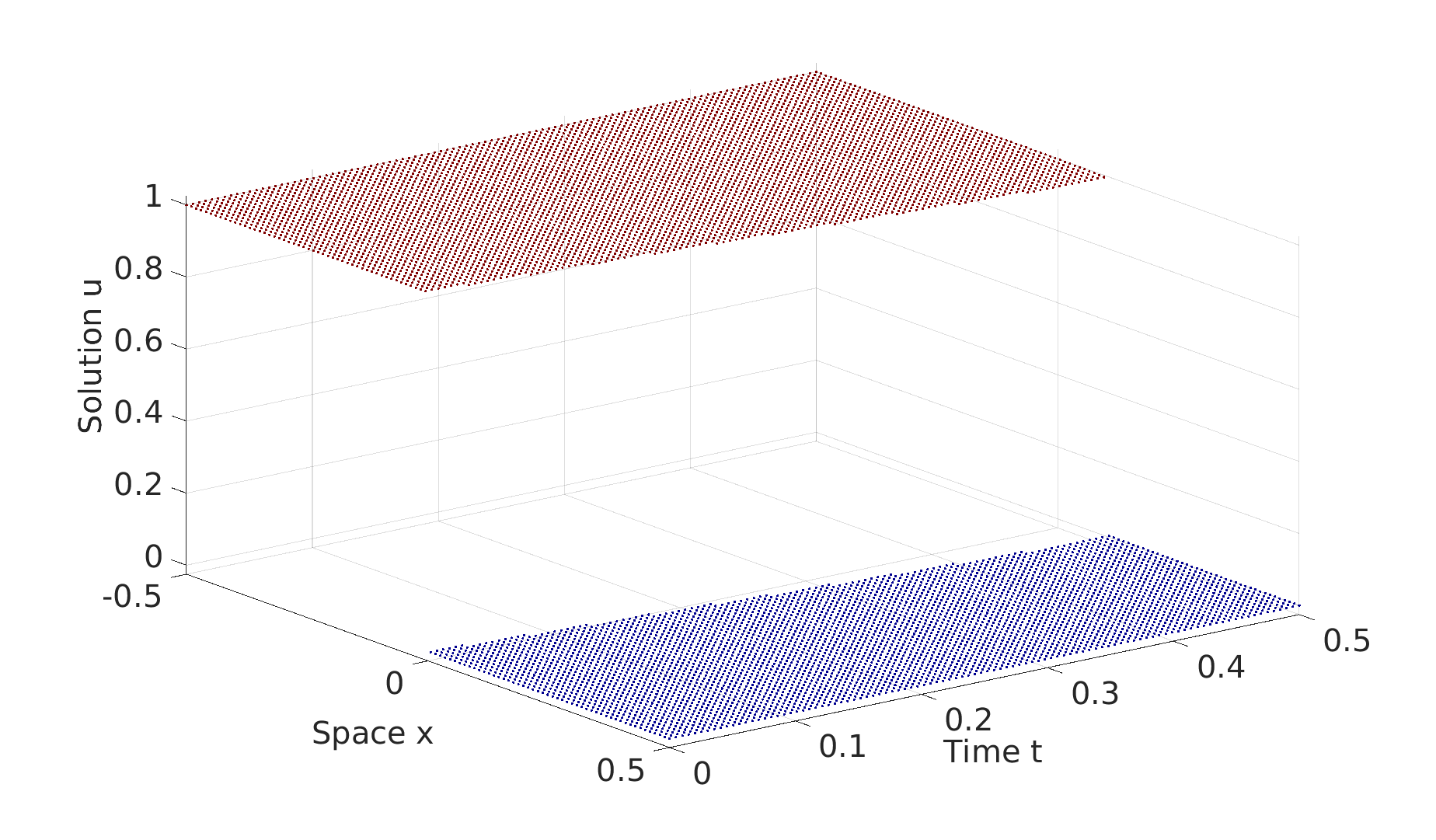}
        \caption{$d=8$}
        \label{fig:fd8xi0}
    \end{subfigure}
    \caption{Graphs of the approximate solution $\widetilde{u_d}(t,x,0)$ for $d=2,5,8$}
    \label{fig:fxi0}
\end{figure}

\reviewde{Figure \ref{fig:fluVar} shows the graphs of the approximate solution $\widetilde{u_d}(\frac14,x,\xi)$ for $x\in\mathbf X$ and \(\xi=0,1\) (so that the speed of the shock is $\frac14$ and \(\frac12\) respectively), with hierarchy's degree $d=2,5,8$, superposed with the exact solution.} \reviewde{Once more, we see that our approximation is very close to the exact solution for \(d=5\), and indistinguishable from the exact solution  for \(d=8\).}

\begin{figure}[H]
    \centering
    \begin{subfigure}[b]{0.49\textwidth}
        \centering
        \includegraphics[width=\textwidth]{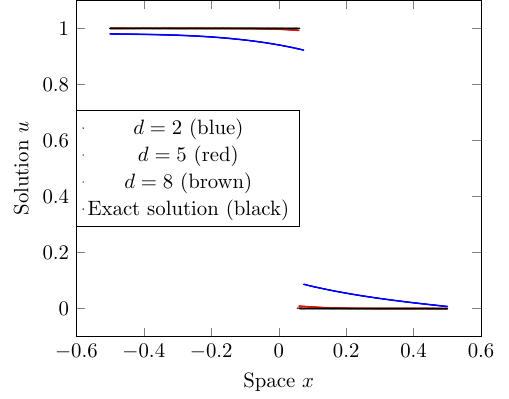}
        \caption{\reviewde{\(\xi=0\)}}
        \label{fig:fluVarXi0}
    \end{subfigure}
    \begin{subfigure}[b]{0.49\textwidth}
        \centering
        \includegraphics[width=\textwidth]{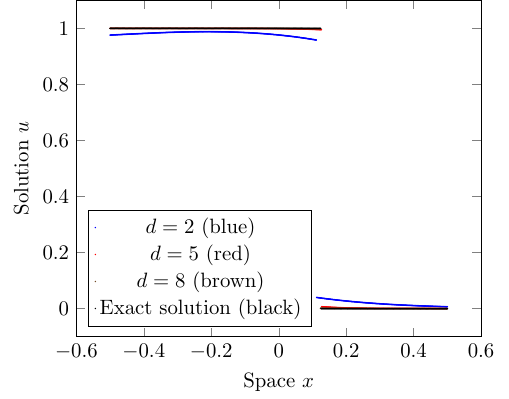}
        \caption{\reviewde{\(\xi=1\)}}
        \label{fig:fluVarXi1}
    \end{subfigure}
    \caption{\reviewde{Graphs of the approximate solution $\widetilde{u_d}(\frac14,x,\xi)$ for $d=2,5,8$, superposed with the exact solution}}
    \label{fig:fluVar}
\end{figure}



\paragraph{$l^1(\mathbf T\times\mathbf X\times\mathbf\Xi)$ error} We pick at random $25$ values in $\mathbf\Xi$, and consider $25$ equidistant values in $\mathbf T$ and $\mathbf X$. We denote the test sets $\mathbf\Xi_e$, $\mathbf T_e$ and $\mathbf X_e$ respectively. We study the evolution of the relative $l^1$ error with respect to the degree $d$ of the hierarchy. Namely, we are interested in
\begin{equation*}
    e_g(d):=\frac{\Vert u-\widetilde{u_d}\Vert_{l^1(\mathbf T_e\times\mathbf X_e\times\mathbf\Xi_e)}}{\Vert u\Vert_{l^1(\mathbf T_e\times\mathbf X_e\times\mathbf\Xi_e)}}.
\end{equation*}
The results are shown in Table \ref{tab:err_par_flux} and on Figure \ref{fig:ploGloErrFlu}.

\begin{table}[h]
    \centering
    \begin{tabular}{c|c c c c c c c}
        d & 2 & 3 & 4 & 5 & 6 & 7 & 8 \\
        \hline
        \reviewde{$q_d$} & \reviewde{140} & \reviewde{420} & \reviewde{990} & \reviewde{2002} & \reviewde{3640} & \reviewde{6120} & \reviewde{9690} \\
        \hline
        $e_g(d)$ & 0.0738 & 0.0285 & 0.0142 & 0.00772 & 0.00780 & 0.00818 & 0.00963
    \end{tabular}
    \caption{Number of unknowns $q_d$ and error $e_g(d)$ for $d=2,\dots,8$}
    \label{tab:err_par_flux}
\end{table}

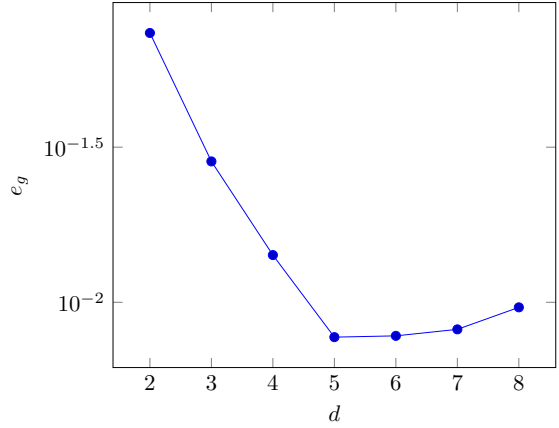
\begin{figure}[H]
    \centering
    \begin{tikzpicture}[scale = .85]
        \begin{semilogyaxis}[
            xlabel={$d$},
            ylabel={$e_g$},
        ]
        
        \addplot
            coordinates {
            (2,0.073821514212147)(3,0.028452982997209)(4,0.014196672906510)(5,0.007721476284434)(6,0.007795082974284)(7,0.008179480123524)(8,0.009630964127542)
            };
            
        \end{semilogyaxis}
    \end{tikzpicture}
    \caption{\reviewde{Evolution of the error $e_g$ with relaxation degree $d$}}
    \label{fig:ploGloErrFlu}
\end{figure}

  We note here the same phenomenon as for the errors presented above occurring, where the approximation improves as $d$ rises until $d=5$. Then the convergence is not monotone and rather slow. 

\paragraph{\reviewde{Conservation condition}} \reviewde{We once more study the conservation condition by plotting the function $t\mapsto 
c_d(t,\xi)=\int_{\mathbf X}(\widetilde{u_d}(t,x,\xi)-u(t,x,\xi))dx$ for $\xi=0.2$ for $11$ equidistant points in $\mathbf T$ and $d=2,5,8$ on Figure \ref{fig:fluCon}.}
\reviewde{The same remark as for the previous experiment holds: the conservation condition is rather well satisfied and it tends to improve as $d$ increases.} 

\begin{figure}[H]
    \centering
    \begin{subfigure}[b]{0.45\textwidth}
        \centering
    \includegraphics[scale=0.8]{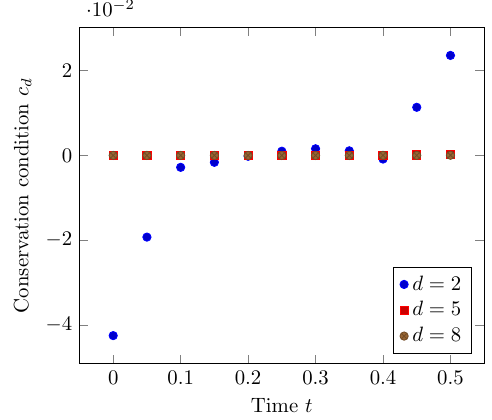}
    \end{subfigure} \hspace{2em}
        \begin{subfigure}[b]{0.45\textwidth}
        \centering
    \includegraphics[scale=0.8]{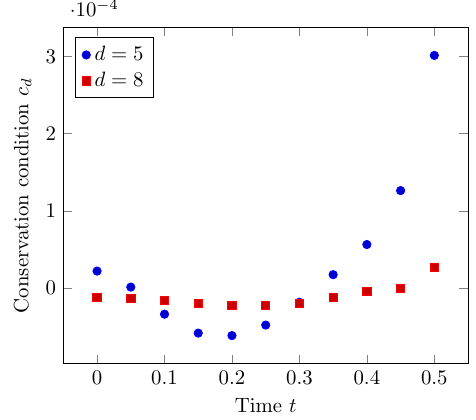}
    \end{subfigure}
    \caption{\reviewde{Conservation condition $c_d(t,0.2)$ versus time $t$ for $d=2,5,8$  (left) and 
     for  $d=5,8$ (right) } }\label{fig:fluCon}
\end{figure}


\clearpage 

\appendix
\appendixpage
\section{Imposing marginal constraints of occupation measures}
\label{ap:impConMomNu}
First, to ensure that the marginal of $\nu$ with respect to $t$, $\mathbf x$ and $\bm\xi$ is the tensor product of the Lebesgue measure on $\mathbf T\times\mathbf X$ and $\rho$, it suffices to impose that
\begin{equation}\label{eq:1bouMea}
\int_{\mathbf K}t^{\alpha_1}\mathbf x^{\bm\alpha_2}\bm\xi^{\bm\alpha_3}d\nu(t,\mathbf x,\bm\xi,y)=\int_{\mathbf T\times\mathbf X\times\mathbf\Xi}t^{\alpha_1}\mathbf x^{\bm\alpha_2}\bm\xi^{\bm\alpha_3}dtd\mathbf xd\rho(\bm\xi),
\end{equation}
for all $\bm\alpha\in\mathbb N^{n+p+1}$. In a similar manner, we impose the marginals of the time boundary measures to be products of a Dirac measure, a Lebesgue measure and $\rho$ as follows: for all $\bm\alpha\in\mathbb N^{n+p+1}$,
\begin{equation}
\int_{\mathbf K}t^{\alpha_1}\mathbf x^{\bm\alpha_2}\bm\xi^{\bm\alpha_3}d\nu_0(t,\mathbf x,\bm\xi,y)=0^{\alpha_1}\int_{\mathbf X\times\mathbf\Xi}\mathbf x^{\alpha_2}\bm\xi^{\alpha_3}d\mathbf xd\rho(\bm\xi),
\end{equation}
\begin{equation}
\int_{\mathbf K}t^{\alpha_1}\mathbf x^{\bm\alpha_2}\bm\xi^{\bm\alpha_3}d\nu_T(t,\mathbf x,\bm\xi,y)=T^{\alpha_1}\int_{\mathbf X\times\mathbf\Xi}\mathbf x^{\bm\alpha_2}\bm\xi^{\bm\alpha_3}d\mathbf xd\rho(\bm\xi),
\end{equation}
and
\begin{equation}\label{eq:lasMarCon}
\int_{\mathbf K}t^{\alpha_1}\mathbf x^{\bm\alpha_2}\bm\xi^{\bm\alpha_3}d\nu_{B,i}(t,\mathbf x,\bm\xi,y)=\int_{\mathbf T\times\Gamma_{B,i}\times\mathbf\Xi}t^{\alpha_1}\mathbf x^{\bm\alpha_2}\bm\xi^{\bm\alpha_3}dtd \mathbf x d\rho(\bm\xi),
\end{equation}
for all $1\le i \le n$ and $B\in \{L,R$\}.
\section{Split measures and corresponding moments constraints}\label{ap:defSplMea}
In addition to split measures $\vartheta_T^+$ and $\vartheta_T^-$ associated with $\nu$, we introduce the time boundary measures $\vartheta_0^+$, $\vartheta_0^-$, $\vartheta_T^+$ and $\vartheta_T^-$ , which are defined as
\begin{subequations}\label{eq:2KruEnt}
\begin{equation}
d\vartheta_0^\pm(t,\mathbf x,\bm\xi,y,v):=\mathbbm 1_{\lbrace u\in\mathbf U:\pm(y-u)\geq 0\rbrace}(v)d\nu_0(t,\mathbf x,\bm\xi,y)dv,
\end{equation}
\begin{equation}
d\vartheta_T^\pm(t,\mathbf x,\bm\xi,y,v):=\mathbbm 1_{\lbrace u\in\mathbf U:\pm(y-u)\geq 0\rbrace}(v)d\nu_T(t,\mathbf x,\bm\xi,y)dv,
\end{equation}
\end{subequations}
with supports
\begin{align}
\mathbf K_0^\pm:=\supp(\vartheta_0^+)=\lbrace(t,\mathbf x,\bm\xi,y,v)\in\mathbf K_0\times\mathbf U:\pm(y-v)\geq 0\rbrace,\\
\mathbf K_T^\pm:=\supp(\vartheta_T^+)=\lbrace(t,\mathbf x,\bm\xi,y,v)\in\mathbf K_T\times\mathbf U:\pm(y-v)\geq 0\rbrace,
\end{align}
respectively.
We only introduce the space boundary measures $(\vartheta_{L_i}^+)_{i=1}^n$, $(\vartheta_{L_i}^-)_{i=1}^n$, $(\vartheta_{R_i}^+)_{i=1}^n$ and $(\vartheta_{R_i}^-)_{i=1}^n$, defined as
\begin{equation}\label{eq:3KruEnt}
d\vartheta_{B,i}^\pm(t,\mathbf x,\bm\xi,y,v):=\mathbf 1_{\lbrace u\in\mathbf U:\pm(y-u)\geq 0\rbrace}(v)d\nu_{L_i}(t,\mathbf x,\bm\xi,y)dv,
\end{equation}
for $B\in\{L,R\}$ and $1\leq i\leq n$, with supports
\begin{align}
\mathbf K_{B,i}^\pm:=\supp(\vartheta_{B,i}^\pm)=\lbrace(t,\mathbf x,\bm\xi,y,v)\in\mathbf K_{B,i}\times\mathbf U:\pm(y-v)\geq 0\rbrace.
\end{align}
The relation between $\nu$ and split measures $\vartheta^+$ and $\vartheta^-$ is imposed through moment constraints
\begin{equation}\label{eq:comConTheNu}
    \int_{\mathbf K\times\mathbf U}\mathbf w^{\bm\alpha}v^{\beta}(d\vartheta^++\reviewun{d}\vartheta^-)(\mathbf w,v)=\int_{\mathbf K\times\mathbf U}\mathbf w^{\bm\alpha}v^{\beta}d\nu(\mathbf w)dv,
\end{equation}
for all $\bm\alpha\in\mathbb N^{n+p+2}$ and for all $\beta\in\mathbb N$. Similar conditions are imposed between time and boundary measures and their corresponding split measures.

\section{\reviewde{Not imposing \eqref{eq:couConLaw} for the Burgers equation}}\label{ap:no_dyn}
\reviewde{For illustrative purposes, we solved the GMP and reconstructed the approximate solution without imposing \eqref{eq:couConLaw} for the Riemann problem where the initial position of the shock is parametrised. At relaxation degree $d=5$ and for $\xi=0.6$, which is equivalent to a shock initially located at $x=-\frac1{10}$, it yields Figure \ref{fig:no_dyn}. We observe here that removing \eqref{eq:couConLaw} from the constraints highly degrades the approximation.
}

\begin{figure}[h]
    \centering
    \includegraphics[width=.45\textwidth]{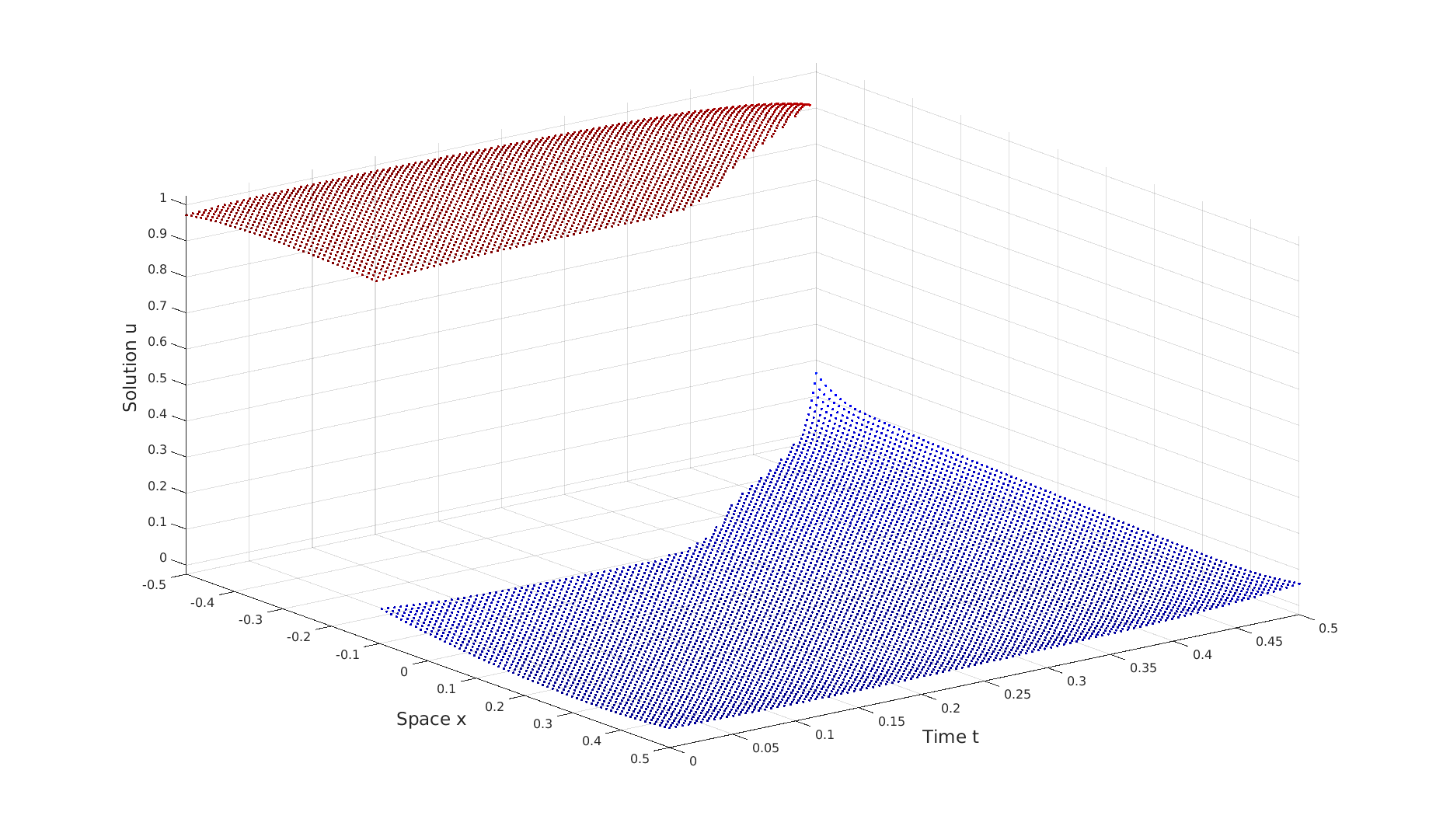}
    \caption{\reviewde{Graph of the approximate solution $\widetilde{u_d}(t,x,0.6)$ for $d=5$ without imposing \eqref{eq:couConLaw}}}
    \label{fig:no_dyn}
\end{figure}

\bibliography{Bibliographie}

\end{document}